\documentclass[12pt]{amsart}
\usepackage{amsmath}
\usepackage{amssymb}
\usepackage{amsthm}
\usepackage{cancel}
\usepackage{tcolorbox}
\usepackage[margin=1.2in]{geometry}
\usepackage{tikz,tikz-cd}
\usetikzlibrary{external}
\usepackage{cases}
\usepackage[]{mdframed}
\usepackage{url}
\usepackage{hyperref}

\newcommand{\N}{\mathbb{N}}
\newcommand{\Z}{\mathbb{Z}}

\newcommand{\R}{\mathbb{R}}

\newcommand{\cat}[1]{\mathrm{CAT}(#1)}
\newcommand{\isom}[1]{\mathrm{Isom}(#1)}

\newcounter{EQNR}
\renewcommand{\contentsname}{Contents\\
{\footnotesize\normalfont(A table of contents should normally not
be included)}}

\theoremstyle{plain}
\newtheorem{thm}{Theorem}[section]

\newtheorem{lem}[thm]{Lemma}
\newtheorem{prop}[thm]{Proposition}
\newtheorem{cor}[thm]{Corollary}

\newtheorem{conjecture}[thm]{Conjecture}

\theoremstyle{definition}
\newtheorem{dfn}{Definition}

\theoremstyle{remark}
\newtheorem{rem}{Remark}

\numberwithin{equation}{section}
\numberwithin{dfn}{section}
\numberwithin{thm}{section}
\numberwithin{rem}{section}

\title[Torsion and fixed-point rigidity]{ Torsion subgroups and fixed-point rigidity \\ in  CAT(0) geometry  }

\author{Jiankun Hou}
\address{Qiuzhen College, Tsinghua University, Beijing, P. R. China}
\email{houjk21@mails.tsinghua.edu.cn}

\author{Hiroyasu Izeki}
\address{Department of Mathematics, Keio University, Kohoku-ku, Yokohama 223-8522, Japan}
\email{izeki@math.keio.ac.jp}

\author{Ran Ji}
\address{Academy for Multidisciplinary Studies, Capital Normal University, Beijing 100048,
P. R. China}
\email{jiran.jbm@gmail.com}

\author{Anders Karlsson 
}
\address{Section de mathématiques, University of Geneva, rue du Conseil-Général 7-9, 1205 Geneva, Switzerland;  Mathematics Department, Uppsala University, Box 256, 751 05, Uppsala, Sweden}
\email{anders.karlsson@unige.ch}

\author{Yunhui Wu}
\address{
Tsinghua University \& BIMSA, Beijing 100084, P. R. China}
\email{yunhui\_wu@mail.tsinghua.edu.cn}

\date{\today}

\begin{document}

\begin{abstract}
We develop new methods for studying groups acting on CAT(0) spaces, which lead to several general structural results.

First, we prove that every torsion subgroup of a CAT(0) group is finite, resolving a question of Swenson from the 1990s. The proof is based on showing 
that random walks on any finitely generated torsion group with bounded exponent acting on a CAT(0) space have zero drift. 
This is then combined with the fixed-point rigidity that we develop.

Second, we show that any finitely generated torsion group of bounded exponent has a global fixed point whenever it acts properly by isometries on a CAT(0) space of bounded geometry, or, without the properness assumption, by isometries on a finite‑dimensional CAT(0) space.

Third, we establish a Kazhdan‑type rigidity principle that underlies many of our results: let $\Gamma$ be a finitely generated group such that every isometric action of $\Gamma$ on $\mathbb{R}^n$ has a fixed point. Then every fixed‑point‑free action of $\Gamma$ on a geodesically complete $n$-dimensional CAT(0) space of bounded geometry has joint minimal displacement uniformly bounded away from zero. In particular, almost fixed points imply a global fixed point. This applies in particular to groups with property (T), torsion groups, certain branch groups, and mapping class groups.

Fourth, we establish the following alternative for any finitely generated amenable group: either every action on a finite‑dimensional CAT(0) space has a global fixed point, or the group has non‑vanishing virtual first Betti number.

Further consequences include that finitely generated torsion groups cannot act without a global fixed point on geodesically complete CAT(0) spaces of bounded geometry that are either visibility spaces or have compact Tits boundary. The methods involve mechanisms that upgrade almost fixed points to genuine fixed points, and rely on Euclidean fixed‑point properties, scalings of actions by ultralimits, and random walks.
\end{abstract}
\maketitle


\section{Introduction}
\subsection{Main results} \label{ssec:mainresults}
Determining which finitely generated groups admit non-degenerate isometric actions on various classes of complete CAT(0) spaces is a central problem in Riemannian geometry and geometric group theory.
For example, the classical theorems of Preissmann and Myers in Riemannian geometry demonstrated how curvature strongly constrains the algebraic structure of fundamental groups. Together with Cartan's fixed point theorem, these results suggest a strong tension between torsion and nonpositive curvature, 
but the extent to which torsion groups can act on CAT(0) spaces has remained an important open problem. In group theory, an early example is Serre’s \emph{property} FA, the fixed-point property for actions on trees, which is closely related to amalgamated products.

The seminal rigidity theory of Mostow and Margulis concerns groups acting by isometries on symmetric spaces of nonpositive curvature. The representation-theoretic property (T) of Kazhdan plays an important role in that theory and in many other areas as well.
This property can in turn be reformulated as \emph{property} FH, the fixed point property on Hilbert spaces. 

In this paper, we develop a new set of methods for the study of isometric actions on spaces of nonpositive curvature. Two fundamental ingredients are random walks and the scaling of actions via ultralimits. 

A \emph{CAT(0) group} is a discrete group that admits a proper and cocompact action on a proper $\cat{0}$ space. The following theorem asserts that every infinite subgroup of a $\cat{0}$ group contains an infinite order element. 

\begin{thm} \label{cor:Swenson-intro}
     Every torsion subgroup of a $\cat{0}$ group is finite.
\end{thm}

This answers a longstanding question first raised by Swenson and subsequently included in several open problem lists in geometric group theory, see \cite{bestvina, caprace-2014, norin-osajda-przytycki}. For example, \cite[Problem 24]{kapovich} refers to it as a conjecture of Swenson. To the best of our knowledge, this problem remained open even in the case where the CAT(0) group is the orbifold fundamental group of a compact nonpositively curved Riemannian orbifold.
The case of dimension 1 is treated in \cite[section 6.5]{serre} and for cube complexes it is established in  \cite{sageev}, without the need for cocompactness. Norin, Osajda, and Przytycki \cite{norin-osajda-przytycki} recently proved it for groups acting on $2$-dimensional CAT(0) complexes, without requiring an embedding into a CAT(0) group.

Our second main result, without the cocompactness assumption, is as follows.
\begin{thm} \label{thm: torsion-intro}
    Let $\Gamma$ be a finitely generated group for which there exists an integer $N$ such that $g^N=e$ for all $g\in \Gamma$. Then we have
     \begin{itemize}
 \setlength{\itemsep}{2pt}
     \item[{\rm (1)}]  Whenever $\Gamma$ acts properly on a complete $\cat{0}$ space of bounded geometry, it must be a finite group.
     \item[{\rm (2)}] Whenever $\Gamma$ acts on a complete finite-dimensional $\cat{0}$ space $Y$, there is a global fixed point in $Y$.
 \end{itemize}
\end{thm}
Part (1) of Theorem \ref{thm: torsion-intro}, whose proof relies on the generalized Margulis lemma due to Breuillard-Green-Tao \cite{breuillard-green-tao},  can be viewed as a generalization of Theorem \ref{cor:Swenson-intro}. Part (2) of Theorem \ref{thm: torsion-intro} states that every finitely generated group acting on a finite-dimensional CAT(0) space without a global fixed point possesses elements of arbitrarily large orders. This result partially addresses a conjecture, as formulated, for instance, in \cite[Conjecture 1.5]{norin-osajda-przytycki}, which asserts that \emph{every finitely generated group acting on a finite-dimensional CAT(0) complex without a global fixed point contains an element of infinite order}. It also implies that Burnside groups cannot act without fixed points on finite-dimensional complete $\cat{0}$ spaces, in contrast to Osajda's construction of  fixed point free actions of the free Burnside groups on certain infinite-dimensional $\cat{0}$ cube complexes \cite{osajda}.


Note that Part (2) of Theorem \ref{thm: torsion-intro}, like most of the results below, does not require the action to be discrete, nor the space to be a complex or admit a cocompact group of isometries.

The proofs of Theorems \ref{cor:Swenson-intro} and \ref{thm: torsion-intro} are based on the key Lemma \ref{lem:block-swap}, which shows that random walks on groups of bounded exponent must have zero drift in $\cat{0}$ spaces, and the application of certain fixed-point rigidity principles that we now describe.\\

The third main result is a rigidity statement for CAT(0) spaces that can be thought of as a nonlinear analogue of Kazhdan’s property (T). 
Sometimes it can serve as a replacement for a generalized Margulis lemma, like the one of Breuillard-Green-Tao \cite[Cor. 1.15]{breuillard-green-tao}, when the action is not properly discontinuous, see subsections \ref{ssec:BGT} and \ref{ssec:visibility} for illustrations of this point. 

Let $\Gamma$ be a group generated by a finite set $S$ and let $\rho$ be an action on a CAT(0) space $Y$ by isometries. Define the \emph{joint minimal displacement} of this action as
\begin{equation*}
 \delta (\rho)=\inf_{x \in Y} \max_{s \in S} d(x, \rho(s)x).
\end{equation*}
We say that the action has {\it almost fixed points} if $\delta (\rho )=0$. Note that this is a property of the action and is independent of the chosen finite generating set. 

A group is said to have {\it property} $\mathrm F \mathbb R ^n$, or the {\it fixed-point property for $\R^n$}, if any isometric action on $\mathbb R^n$ has a global fixed point. It is immediate that $\mathrm F \mathbb R ^n$ implies $\mathrm F \mathbb R ^m$ if $n>m$. A group is said to have {\it property} $\mathrm F \mathbb R ^{<\infty}$ if it has property $\mathrm F \mathbb R ^n$ for every $n \in \mathbb{N}$. In particular, property FH implies property $\mathrm F \mathbb R ^{<\infty}$.

\begin{thm}\emph{($\cat{0}$ Kazhdan constants)}\label{thmD}
Let $\Gamma$ be a group generated by a finite set $S$ and $\mathcal{Y}$ a family of geodesically complete $\cat{0}$ spaces with uniformly bounded geometry. 
 Suppose either
 \begin{itemize}
 \setlength{\itemsep}{2pt}
     \item[{\rm (i)}] that $\Gamma$ has property $\mathrm F \mathbb R ^{<\infty}$, or
     \item[{\rm (ii)}] that there exists a positive integer $n$ such that $\Gamma$ has property $\mathrm F \mathbb R ^{n}$ and $\dim Y \leq n$ for all $Y \in \mathcal{Y}$. 
 \end{itemize}
 Then there is a constant $\kappa (S, \mathcal{Y})>0$ such that, for any $Y\in \mathcal{Y}$ and homomorphism $\rho :\Gamma \rightarrow \mathrm{Isom}(Y)$ that provides a fixed-point-free action, the inequality
 $\delta(\rho ) \geq \kappa (S ,\mathcal{Y})$ holds.
\end{thm}

In particular, every action of $\Gamma$ with almost fixed points must have a global fixed point. In analogy with the classical linear theory, we call $\kappa (S, Y)$ a \emph{$\cat{0}$ Kazhdan constant}.
The assumption of bounded geometry is necessary, see Remark \ref{remark-boundedgeometry}.
We also note here that a geodesically complete $\cat{0}$ space with bounded geometry is finite-dimensional, according to \cite[Theorem 4.9]{cavallucci-sambusetti}.

Theorem \ref{thmD} applies to many groups of a variety of origins, from Lie group theory, algebra, and topology.\\


The fourth main theorem is the following:

\begin{thm} \emph{(CAT(0) alternative for amenable groups)} \label{thm:amenable-intro}
  Let $\Gamma$ be a finitely generated amenable group.  Then exactly one of the following holds:
 \begin{itemize}
  \item[{\rm (1)}]  every isometric action of $\Gamma$ on 
   a complete finite-dimensional $\cat{0}$ space
  has a global fixed point;
  \item[{\rm (2)}] $\Gamma$ admits an isometric action on some $\R^n$ without a global fixed point, or equivalently, $\Gamma$ has nonvanishing virtual first Betti number.
 \end{itemize}
 Moreover, the minimal dimension of a complete $\cat{0}$ space admitting a fixed-point-free action by $\Gamma$ equals the minimal dimension of a Euclidean space admitting such an action.
 \end{thm}

In infinite dimensions, every infinite finitely generated amenable group acts on a Hilbert space without a global fixed point \cite{bekka-cherix-valette}. 
Our theorem significantly generalizes \cite{izeki-karlsson} and should be compared with the foundational theorems by Burger-Schroeder \cite{BS87}, Adams-Ballmann \cite{AB98} and Caprace-Lytchak \cite{caprace-lytchak}, which state that actions by amenable groups either stabilize a flat or fix a boundary point at infinity. 
Theorem \ref{thm:amenable-intro} strengthens this partly by identifying Euclidean actions as the unique obstruction to global fixed points in the space for amenable groups acting 
on finite-dimensional $\cat{0}$ spaces.
Amenable group actions on CAT(0) spaces have been studied extensively by Caprace and Monod, often under the assumption that the full isometry group acts cocompactly, see \cite{caprace, caprace-monod1, caprace-monod2}. For interesting related results in the case of CAT(0) cube complexes, we refer to the recent paper by Genevois \cite{genevois}.

\subsection{A selection of corollaries}
\subsubsection{Torsion subgroups of CAT(0) groups}
Theorem \ref{cor:Swenson-intro} asserts that every infinite subgroup of a $\cat{0}$ group contains an infinite order element. 
Swenson's original interest in this question was that assuming this was true, he could prove that one-ended $\cat{0}$ groups have no cut-point in the boundary. The latter was subsequently proved by Papasoglu-Swenson in \cite{papasoglu-swenson} using another approach. For a larger perspective on these types of results, we refer to \cite{haettel-osajda}.

A further consequence of Theorem \ref{cor:Swenson-intro} is that it removes the assumption that torsion subgroups are finite in some of the results of M\"oller-Varghese in \cite{moeller-varghese}, in particular their Corollary E, which concerns the continuity of abstract group homomorphisms.

\subsubsection{Finitely generated torsion groups} Since a CAT(0) group admits only finitely many conjugacy classes of finite subgroups, its elliptic elements have uniformly bounded orders (see e.g. \cite{bridson-haefliger}). It suffices for proving Theorem \ref{cor:Swenson-intro} to consider finitely generated torsion subgroups of bounded exponent. Consequently, Part (1) of Theorem \ref{thm: torsion-intro} can be viewed as a generalization of Theorem \ref{cor:Swenson-intro}. We note that the proof of Theorem \ref{cor:Swenson-intro} does not rely on the generalized Margulis lemma due to Breuillard-Green-Tao \cite{breuillard-green-tao}, which is, however, essential in the proof of Part (1) of Theorem \ref{thm: torsion-intro}. This in particular implies that Burnside groups cannot act properly without fixed points on complete $\cat{0}$ spaces of bounded geometry.

\subsubsection{Kazhdan-type rigidity principle}

Theorem \ref{thmD} applies to groups with Kazhdan's property (T), such as higher rank lattices and $\mathrm{Aut}(F_n)$, $n\geq 5$, and groups all of whose finite-dimensional real linear representations have finite image, since the isometry group of a Euclidean space is linear. The latter class includes torsion groups and nonlinear just infinite groups, and there are many such examples within the family of weakly branch groups \cite{abert, bartholdigrigorchuksunik}.
Moreover, since finitely generated linear groups are residually finite, groups with no nontrivial finite quotients, such as infinite simple groups, cannot act on Euclidean spaces without a fixed point. 

Furthermore, many important groups, for example mapping class groups, do not admit fixed-point-free Euclidean actions in small dimensions, see \cite{bridson} for a discussion on related dimension restrictions. 

Theorem \ref{thmD} has many consequences, see section \ref{sec:rigidityconsequences} for the proofs of the following assertions.
It can also be viewed as a nonlinear extension of one part of Weil's celebrated cohomological criterion for local rigidity, \cite{weil}:

\begin{cor}
    Let $\Gamma$ be a finitely generated group with property $\mathrm F \mathbb R ^{<\infty}$. Given a finite generating set, there is a constant $\kappa_n >0$ for each dimension $n$, so that for any representation  $\rho :\Gamma \rightarrow \mathrm{GL}_n(\mathbb{C})$ that is unbounded, $\delta(\rho)\geq \kappa_n$. In particular, any small deformation of a unitary representation of $\Gamma$ among linear representations is unitarizable. 
\end{cor}

 The last statement is a weaker version of the local rigidity of unitary representations coming from Weil's criterion as employed by Rapinchuk in \cite{rapinchuk} for property (T) groups. For recent extensions of the Selberg-Calabi-Weil local rigidity theorems, in particular to a $\cat{0}$ setting, we refer to Gelander-Levit \cite{gelander-levit}.

Another consequence is a drift dichotomy for random walks in nonpositive curvature, see Theorem \ref{thm:F}. Here is an illustration of this in a special setting of considerable importance in rigidity theory and dynamical systems, see \cite[§3]{margulis}, \cite[Sect. 4.7]{benoist-quint}, and \cite{bochinavas}. Note that in contrast to the standard results, see for example \cite[Theorem 4.31]{benoist-quint}, our compactness-versus-positive Lyapunov exponent dichotomy needs no (strong) irreducibility or proximality assumption, instead requiring a group theoretic property:

\begin{cor} \label{cor:Lyapunov}
Let $\mu$ be a symmetric, Borel probability measure on $SL_n(\mathbb{R})$ with finite second moment. Assume that the group generated by the support of $\mu$ is a finitely generated group with property $\mathrm F \mathbb R^{n(n+1)/2-1}$, and that it is unbounded in $SL_n(\mathbb{R})$.  Then the top Lyapunov exponent $\lambda_{1,\mu}$ is strictly positive.  
\end{cor}

Note that if the orbit is unbounded, Theorem \ref{thmD} gives a gap for the joint minimal displacement. This can then be made into a condition on the spectral radius on a finite number of individual matrices for when the top Lyapunov exponent is positive, using \cite[Lemma 1.7]{breuillard-fujiwara} and an inequality due to Bochi, see \cite{bochi, breuillard-fujiwara}.

Another consequence of Theorem \ref{thmD} is
a result for torsion groups that complements \cite[Theorem 1.5]{ji-wu_tits_alt} since we do not assume that the action is properly discontinuous, on the other hand, we insist on geodesic completeness, which is not assumed in \cite{ji-wu_tits_alt}.
\begin{cor}
     Let $Y$ be a geodesically complete visibility space
 with bounded geometry, and let $\Gamma$ be a finitely generated torsion
 group. Whenever $\Gamma$ acts isometrically on $Y$, there is a global fixed point in $Y$.
\end{cor}

The proof of this last corollary in subsection \ref{ssec:visibility} illustrates how Theorem \ref{thmD} 
can sometimes be used when no generalized Margulis lemma is applicable.

\subsubsection{Amenable groups}

Since the isometry groups of Euclidean spaces are linear, the following is an immediate corollary of Theorem \ref{thm:amenable-intro} that responds partially to a general question raised by Margulis in \cite[§2, Problem 4]{margulis} about the linear nature of groups acting on Cartan-Hadamard manifolds:

\begin{cor}
   Any finitely generated amenable group that admits a fixed-point free action on a finite-dimensional complete CAT(0) space must possess an infinite-image, finite-dimensional linear representation. 
\end{cor}

This leads to (Corollaries \ref{cor:C} and \ref{cor:D}):
\begin{cor} 
Any finitely generated amenable group, which is either 
\begin{itemize}
 \setlength{\itemsep}{2pt}
     \item torsion, or
     \item virtually simple, or 
     \item nonlinear and just infinite,
 \end{itemize}
must fix a point whenever acting on finite-dimensional complete $\cat{0}$ spaces.
\end{cor}

Groups of these three types that are infinite exist and have generated a great deal of interest, see the references in section \ref{sec:an-altern-amen}.

\subsection{Methods.}
This paper contains several new tools and novel combinations of these techniques with more standard methods. We feel that their power is well illustrated by the results obtained and we believe that they will be useful for other problems as well.

Note that contrary to many existing results, our methods do not require the action to be proper or discrete, nor do they rely on the presence of a large isometry group of the space. The main theorems and most of their consequences below are new even in the setting of proper actions.

One fundamental methodology involves random walks, even if all main results are entirely nonrandom. The zero drift case is handled by the paper of Izeki \cite{Izeki_2023}, in the form of Theorem \ref{thm:IK23} below (harmonic maps play an important role there), and for the positive
drift case, the main theorem of Karlsson-Margulis \cite{karlsson-margulis}
is applied in several ways. 

The second principal device is rescaling of the actions and metric spaces using ultra limits, for example in combination with random walks as in the proof of Theorem \ref{thm: torsion-intro}, or for Theorem \ref{thm:amenable-intro} using an argument due to Gromov (see the first paragraph of section \ref{sec:an-altern-amen}). In the proof of that latter theorem 
a result of Caprace-Lytchak \cite{caprace-lytchak} is crucial.

The block-swap lemma, Lemma \ref{lem:block-swap}, is perhaps the single most original argument.

Another novel ingredient is a mechanism for obtaining compactness of Tits boundaries and exploiting this property. 

Note that although we make extensive use of the visual boundary at infinity and Busemann functions, as is standard in the subject, all of our main results concern fixed points in the space itself rather than on the boundary. This contrasts with some parts of the existing literature.

\section{CAT(0) geometry}
\label{sec:resol-bound-fixed}
In this section, we make some preparations.
\subsection{A resolution of boundary fixed points}
Let $\omega$ be a non-principal ultrafilter on $\mathbb{N}$. Given a sequence of complete $\cat{0}$ spaces $Y_n$ with basepoints $p_n$, the associated ultralimit $(Y_\infty,p_\infty)=\omega\text{-}\lim (Y_n,p_n)$ is a complete $\cat{0}$ space \cite[Chapter II.3 Theorem 3.9]{bridson-haefliger}. If the dimensions of the spaces $Y_n$ are at most $D$, then the dimension of $Y_\infty$  is also at most $D$ \cite[Proposition 6]{izeki-karlsson}.

The following ultralimit space $Z$, associated with an isometric group action $\Gamma$ on a complete $\cat{0}$ space $Y$, is fundamental and useful in this work. One may consult \cite[Proposition 2.8]{Leeb} and \cite[Proposition 3.3]{caprace-monod1} for constructions without using ultralimits.
\begin{lem}
\label{lem:ultralimit}
Let $Y$ be a complete $\cat{0}$ space, and let $\Gamma$ be a finitely generated group acting on $Y$ via a homomorphism $\rho:\Gamma \rightarrow \isom{Y}$. Suppose that 
$\rho(\Gamma)$ fixes a point $\xi$ in $\partial Y$. Let $c:[0,\infty] \to Y$ be a geodesic ray with $c(\infty)=\xi$. Then the ultralimit 
 \begin{equation*}
 (Z,o')=\omega\text{-}\lim (Y,c(n))
 \end{equation*} 
 is a complete $\cat{0}$ space on which $\Gamma$ acts via a homomorphism $\rho':\Gamma \rightarrow \isom{Z}$ satisfying 
 \begin{equation*}
 \inf_{q \in Z} \max_{s \in S} d(q, \rho'(s) q)\geq \inf_{p \in Y} \max_{s \in S} d(p, \rho(s) p)
 \end{equation*}
 for any finite generating set $S \subset \Gamma$.
Moreover, $Z$ contains a $\rho'(\Gamma)$-invariant convex subset $Z'$ that splits isometrically as $Z_0 \times \mathbb{R}$, on which the action of $\rho'(\Gamma)$ preserves the splitting.
 \end{lem}
 \begin{proof}
Without loss of generality, we may assume that $c$ has unit speed. Take $o=c(0)$, and denote
 by $c_{\gamma}$ the unit speed geodesic ray starting from $\rho(\gamma)o$ and
 terminating at $\xi$, $\gamma \in \Gamma$. Note that $c_e=c$, where $e$ is the identity element of $\Gamma$.
 
 Since $d(c(n), \rho(\gamma)c(n))\leq d(c(0),c_{\gamma}(0))$, we
 have a well-defined homomorphism 
 $\rho'\colon \Gamma\rightarrow \isom{Z}$.  For any $q=\omega\text{-}\lim p_n\in Z$ with $p_n \in Y$, we have
 \begin{equation*}
  d(q,\rho'(\gamma)q)=\omega\text{-}\lim d(p_n, \rho(\gamma)p_n),
 \end{equation*}
this implies
\begin{equation*}
 \inf_{q \in Z} \max_{s \in S} d(q, \rho'(s) q) \geq \inf_{p \in Y} \max_{s \in S} d(p, \rho(s) p).
\end{equation*}

Let us define $c^{(n)}_{\gamma}\colon \mathbb{R} \rightarrow Y$ by
\begin{equation*}
 c^{(n)}_{\gamma}(t) =
  \begin{cases}
   c_{\gamma}(0) & t \leq -n \\
   c_{\gamma}(t+n) & t> -n.
  \end{cases}
\end{equation*}
 Note that we have, for any $t\in \R$, and $\gamma, \gamma' \in \Gamma$,
 \begin{equation}
 \label{eq:bounded}
 d(c_{\gamma}^{(n)}(t),c_{\gamma'}^{(n)}(t))\leq
  d(c_{\gamma}(0),c_{\gamma'}(0)),  
 \tag{$\star$}
 \end{equation}
 and $c_{\gamma}^{(n)}$ restricted to $[-n,\infty)$ is a geodesic ray. 
 Then
 \begin{equation*}
  \omega\text{-}\lim c_{\gamma}^{(n)}(0)=\omega\text{-}\lim
 c_{\gamma}(n)=\omega\text{-}\lim \rho(\gamma)c_e(n)=\rho'(\gamma)o', 
 \end{equation*} 
 and therefore $c'_{\gamma}\colon t\mapsto \omega\text{-}\lim c_{\gamma}^{(n)}(t)$
 is a geodesic defined on all of $\mathbb{R}$ passing through
 $\rho'(\gamma)o'$. 
 Since 
 $\rho(\gamma')c_{\gamma}^{(n)}(t)=c_{\gamma'\gamma}^{(n)}(t)$, 
 we get 
 \begin{equation*}
 \rho'(\gamma')c'_{\gamma}(t)
 =\omega\text{-}\lim \rho(\gamma')c_{\gamma}^{(n)}(t)
 =\omega\text{-}\lim c_{\gamma'\gamma}^{(n)}(t)
 = c'_{\gamma'\gamma}(t). 
 \end{equation*}

 Recalling (\ref{eq:bounded}), we see that
 $d(c'_{\gamma}(t),c'_{\gamma'}(t))\leq d(c_{\gamma}(0),c_{\gamma'}(0))$
 holds for any $t\in \mathbb{R}$, that is, 
 $c'_{\gamma}$ and $c'_{\gamma'}$ are parallel to each other. 
 We denote by $\xi'$ the point in $\partial Z$ given by
 $c'_{\gamma}(\infty)$ for any $\gamma \in \Gamma$; $\rho'(\Gamma)$
 fixes $\xi'$. 

 Let $Z'$ be the union of the images of geodesics parallel to
 $c'_e$. $Z'$ is a closed convex subset of $Z$ and hence a complete $\cat{0}$ space.
 By \cite[p.~183, 2.14 A Product Decomposition
 Theorem]{bridson-haefliger}, we have a splitting 
 $Z'=Z_0 \times \R$, and $\{z\}\times \R$
 corresponds to a geodesic parallel to $c_{e}'$ for each 
 $z\in Z_0$. This completes the proof.
 \end{proof}

 \begin{lem}
\label{lem:splitting}
 Let $Y$, $\Gamma$, $\xi$ be as in Lemma \ref{lem:ultralimit}. Assume further that
 \[\delta:= \inf_{p \in Y} \max_{s \in S} d(p, \rho(s) p)>0\]
 for some finite generating set $S\subset \Gamma$, that $\Gamma$ admits no nontrivial homomorphism into $\mathbb{R}$, and that $Y$ has finite dimension. Then there exists a complete $\cat{0}$ space $Z_0$ and a homomorphism $\rho_0:\Gamma \rightarrow \isom{Z_0}$ such that
 \vspace{-0.1cm}
\begin{itemize}
 \setlength{\itemsep}{2pt} 
 \item[{\rm (1)}] $\dim Z_0 < \dim Y$;
 \item[{\rm (2)}] $\rho_0(\Gamma)$ does not fix any point in 
	      $Z_0 \cup \partial Z_0$.
\end{itemize}
\end{lem}
\begin{proof}
Take $o \in Y$ and let $c$ be the geodesic ray starting from $o$ and terminating at $\xi$. Consider $(Z,o')=\omega\text{-}\lim (Y,c(n))$. By Lemma \ref{lem:ultralimit}, there exists a complete $\cat{0}$ space $Z'=Z_1 \times \mathbb{R} \subset Z$, on which $\Gamma$ acts via a homomorphism $\rho':\Gamma \to \isom{Z'}$. We have $\dim Z_1 < \dim Z' \leq \dim Y$.

Let $\pi: Z'=Z_1 \times \mathbb{R} \to \mathbb{R}$ be the projection onto the second factor. Since the action of $\rho'(\Gamma)$ on $Z_1 \times \mathbb{R}$ preserves the splitting (fixing the distinguished endpoint thus excluding reflections), the map
\begin{equation*}
\gamma \mapsto \pi(\rho'(\gamma)o')-\pi(o')
\end{equation*}
gives a homomorphism from $\Gamma$ into $\R$. By our assumption on $\Gamma$, this homomorphism must be trivial. It follows that $\rho'(\Gamma)$ induces an action on $Z_1$ (still denoted by $\rho'(\Gamma)$) that satisfies
\begin{equation*}
 \inf_{q \in Z_1} \max_{s \in S} d(q, \rho'(s) q)\geq \delta.
\end{equation*} 
In particular, $\rho'(\Gamma)$ does not have fixed points in $Z_1$.

Suppose that $\rho'(\Gamma)$
fixes a point in $\partial Z_1$. 
We can then repeat the procedure described above to obtain 
a complete $\cat{0}$ space $Z_2$ with $\dim Z_2 < \dim Z_1$, together with a
homomorphism $\rho'':\Gamma \rightarrow \isom{Z_2}$  satisfying
\begin{equation*}
\inf_{q \in Z_2} \max_{s \in S} d(q, \rho''(s) q)\geq \delta,
\end{equation*}
which implies that the action does not fix any point
 in $Z_2$. 
Iterating this process, in a finite number of steps, say $k$, 
we get a $\cat{0}$ space $Z_k$ on which $\Gamma$ acts without any fixed points
in $Z_k \cup \partial Z_k$.  
Otherwise, eventually we reach a one-dimensional complete $\cat{0}$ space $\tilde{Z}$ on
which $\Gamma$ acts without fixed points in $\tilde{Z}$ but  fixes a point in $\partial \tilde{Z}$. Then again by taking an ultralimit $\tilde{Z}'$ of $\tilde{Z}$, we obtain a subset $\tilde{Z}'' \subset \tilde{Z}'$ consisting of the images of parallel geodesic lines. 
Since $\dim \tilde{Z}'' \leq 1$, we have $\tilde{Z}''=\R$. Therefore, the action of $\Gamma$ on $\tilde{Z}''$, either has a fixed point, or induces a nontrivial homomorphism from $\Gamma$ to $\mathbb{R}$, both of which contradict our assumptions on $\Gamma$.
\end{proof}

\begin{rem}
 We cannot remove the assumption that $\Gamma$ does not admit a nontrivial
 homomorphism into $\R$. Indeed, let $g \in \isom{Y}$ be a hyperbolic
 element and $\Gamma =\langle g \rangle \cong \Z$.  Then $\Gamma$
translates every axis of $g$.  By the Caprace-Lytchak theorem (\cite{caprace-lytchak}) we cannot find a finite-dimensional
 $\cat{0}$ space $Z$ and a fixed-point-free action of $\Gamma \cong \Z$ on $Z$
 that does not fix any point in $\partial Z$. 
\end{rem}

\subsection{Bounded geometry and drift} 
Let $X=(X,d)$ be a metric space and let
$\varepsilon>0$. 
A subset $N$ of $X$ is called {\it $\varepsilon$-sparse} if
$d(x,x')\geq \varepsilon$ for all distinct $x,x'\in N$. 
For $A \subset X$, we denote by $n_{\varepsilon}(A)$ the maximal cardinality of an $\varepsilon$-sparse subset of $A$. 
Note that if $N\subset A$ is an $\varepsilon$-sparse subset of maximal
cardinality, then $A\subset \bigcup_{x\in N}B(x,\varepsilon)$, where
$B(x,\varepsilon)$ denotes the open ball centered at $x$ with radius
$\varepsilon$. For $r>0$ and $\varepsilon >0$, we set
\begin{equation*}
 n_{r,\varepsilon}(X)=\sup_{x\in X} n_{\varepsilon}(B(x,r)).  
\end{equation*}

\begin{dfn}[{\cite[\S 4.1]{caprace}},\cite{cavallucci-sambusetti}]
\label{defn:bdd_geometry}
A metric space $X=(X,d)$ is said to have {\it bounded geometry} if, for all
 $r>\varepsilon>0$, $n_{r,\varepsilon}(X)$ is finite. In this case, we call $n_{r,\varepsilon}(X)$ the {\it $(r,\varepsilon)$-capacity} of $X$.
\end{dfn}

We note that if a complete metric space $X$ has bounded geometry, then $X$ is proper (\cite[Lemma
 4.1(i)]{caprace}). 

 \begin{rem} \label{rem:boundedgeometry}
It is known that a $\cat{0}$ space admitting a
geometric group action always has bounded geometry (\cite[Lemma
4.1(ii)]{caprace}).   
A detailed proof can be found, for example, in \cite{toyoda}. 
\end{rem}

Let $\mu$ be a symmetric probability measure on $\Gamma$ with finite first moment, whose support generates $\Gamma$. For any isometric action $\rho:\Gamma \to \isom{X,d}$, the {\it drift $l_{\mu}(\rho)$} of $\rho(\Gamma)$ with respect to $\mu$ is defined as 
\begin{equation*}
 l_{\mu}(\rho):=\lim_{n\to \infty}
 \frac 1n \int_{\Gamma}d(x,\rho(\gamma)x) d\mu^{*n}(\gamma), 
\end{equation*}
 where  $\mu^{*n}$ denotes the $n$th convolution of $\mu$. It is well-known that the limit always exists and is independent of the choice of $x \in X$, see \cite{karlsson-linear} for more information.

The following result is analogous to Lemma \ref{lem:splitting}, where the finite-dimensional condition forces the iteration of ultralimits to terminate. We show that the same conclusion holds for spaces with bounded geometry.
\begin{lem}
  \label{lem:bounded geometry}
 Let $\Gamma$ be a finitely generated group that admits no nontrivial
 homomorphism into $\R$. Suppose that $\Gamma$ acts on 
 a complete $\cat{0}$ space $Y$ with bounded geometry via
 a homomorphism $\rho:\Gamma \rightarrow \isom{Y}$. Then there exists a complete $\cat{0}$ space $Z_0$ and a homomorphism $\rho_0:\Gamma \rightarrow \isom{Z_0}$ such that
 \vspace{-0.1cm}
\begin{itemize}
 \setlength{\itemsep}{2pt} 
 \item[{\rm (1)}] $Z_0$ has bounded geometry;
 \item[{\rm (2)}] $\rho_0(\Gamma)$ does not fix any point in 
	      $\partial Z_0$.
\item[{\rm (3)}] $\rho_0$ satisfies
  \begin{equation*}
   \inf_{q \in Z_0} \max_{s \in S} d(q, \rho_0(s) q)\geq \inf_{p \in Y} \max_{s \in S} d(p, \rho(s) p), 
 \end{equation*}
for any finite generating set $S \subset \Gamma$.
\end{itemize}
 Furthermore, 
 \vspace{-0.1cm}
\begin{itemize}
        \item[{\rm (4)}] $l_{\mu}(\rho_0) \leq l_{\mu}(\rho)$, where both drifts are computed with respect to the same symmetric probability measure with finite first moment.
 \end{itemize}
\end{lem}
\begin{proof}
If $\rho(\Gamma)$ does not fix any point in $\partial Y$, then the lemma is proved by setting $Z_0=Y$ and $\rho_0=\rho$.

It remains to consider the case that $\rho(\Gamma)$ fixes a point $\xi \in \partial Y$.
Let $c:[0,\infty] \to Y$ be a geodesic ray with $c(0)=o$ and $c(\infty)=\xi$. Consider the ultralimit $(Z,o')=\omega\text{-}\lim (Y,c(n))$.

By Lemma \ref{lem:ultralimit}, there exists a complete $\cat{0}$ space $Z'=Z_1 \times \mathbb{R} \subset Z$, on which $\Gamma$ acts via a homomorphism $\rho':\Gamma \rightarrow \isom{Z'}$ that preserves the splitting. Since $\Gamma$ admits no nontrivial homomorphism into $\mathbb{R}$, the action of $\rho'(\Gamma)$ on the $\mathbb{R}$-factor is trivial. Therefore, $\rho'(\Gamma)$ induces an isometric action of $\Gamma$ on $Z_1$, which we still denote by $\rho'$. Write $o'=(o_1,0)$ with $o_1 \in Z_1$.

Repeating the argument in the proof of Lemma \ref{lem:ultralimit} yields that 
\begin{equation*}
\inf_{q \in Z_1} \max_{s \in S} d(q, \rho'(s) q)\geq \inf_{p \in Y} \max_{s \in S} d(p, \rho(s) p).
\end{equation*} 
Moreover, we have $d(o_1,\rho'(\gamma) o_1)=d(o',\rho'(\gamma) o')\leq d(o,\rho(\gamma) o)$ for every $\gamma \in \Gamma$. It follows that
\begin{eqnarray*}
l_{\mu}(\rho')&=& \lim_{n\to \infty} 
 \frac{\int_{\Gamma}d(o_1,\rho'(\gamma)o_1) d\mu^{*n}(\gamma)}{n}\\
 &\leq& \lim_{n\to \infty} 
 \frac{\int_{\Gamma}d(o,\rho(\gamma)o) d\mu^{*n}(\gamma)}{n}=l_{\mu}(\rho).
\end{eqnarray*}

For any $r>r'>\varepsilon'>\varepsilon>0$ and any $q \in Z_1$, let $\{x_1,\cdots,x_N\} \subset B(q,r')$ be an $\varepsilon'$-sparse set in $Z_1$. Set $z_i=(x_i,0) \in Z', 1\leq i \leq N$ and $z_{N+1}=(q,\varepsilon'), z_{N+2}=(q,-\varepsilon') \in Z'$. It is easy to verify that $\{z_1,\cdots,z_{N+2}\}$ is an $\varepsilon'$-sparse set in $B((q,0),r') \subset Z'$. It follows from the definition of ultralimits that we can find in $Y$ an $\varepsilon$-sparse set of cardinality $N+2$, contained in a ball of radius $r$. Since $Y$ has bounded geometry, let $n_{r,\varepsilon}=n_{r,\varepsilon}(Y)$ be its $(r,\varepsilon)$-capacity, then we have
$$N+2 \leq n_{r,\varepsilon},$$ which implies that $Z_1$ has bounded geometry and its $(r',\varepsilon')$-capacity is at most $n_{r,\varepsilon}-2$. 

The space $Z_1$ associated with the action $\rho'$ satisfies all conditions of a complete CAT(0) space required in Lemma \ref{lem:bounded geometry} except possibly condition (2). If $\rho'(\Gamma)$ does not fix any point in $\partial Z_1$, then we are done by setting $Z_0=Z_1$; otherwise, we can iterate the preceding process described above. At step $k \leq n_{r,\varepsilon}$, suppose that there exists a $(\varepsilon+\frac{k(\varepsilon'-\varepsilon)}{n_{r,\varepsilon}})$--sparse set of cardinality $K$ contained in a ball of radius $\left(r-\frac{k(r-r')}{n_{r,\varepsilon}}\right)$ in $Z_k$. Then by the discussion in the preceding paragraph, we can find a $\left(\varepsilon+\frac{(k-1)(\varepsilon'-\varepsilon)}{n_{r,\varepsilon}}\right)$--sparse set of cardinality $K+2$ contained in a ball of radius $\left(r-\frac{(k-1)(r-r')}{n_{r,\varepsilon}}\right)$ in $Z_{k-1}$. Iterating backward to $k=1$ we obtain
\begin{equation*}
n_{r,\varepsilon} \geq K+2k,
\end{equation*}
which yields $$k \leq n_{r,\varepsilon}/2,$$ and hence the iteration terminates in finitely many steps. We eventually obtain a complete $\cat{0}$ space $Z_0$ associated with an isometric action $\rho_0:\Gamma \rightarrow \isom{Z_0}$ that fulfills all conditions $(1)-(4)$.
\end{proof}
\subsection{Existence of almost fixed points}
Let $\Gamma$ be a finitely generated group. 
We say that $\Gamma$ has the {\it fixed-point property for $\R^n$} if any isometric
action of $\Gamma$ on $\R^n$ has a fixed point. Suppose $\Gamma$ acts on a metric space $X=(X,d)$ via a homomorphism 
$\rho \colon \Gamma \rightarrow \isom{X}$. Set
\begin{equation*}
 \delta (\rho):=\inf_{x \in X} \max_{s \in S} d(x, \rho(s)x). 
\end{equation*}
We say that $\rho(\Gamma)$ has {\it almost fixed points} if $\delta(\rho)=0$ for some (hence any) finite generating set $S$ of $\Gamma$. 

The following formulation of a result of Izeki \cite[Theorem A]{Izeki_2023} is essential for the proofs in this subsection.
\begin{thm}
\label{thm:IK23}
Let $Y$ be a complete $\cat{0}$ space that is either proper or of finite dimension, and $\Gamma$ a finitely generated group equipped with a symmetric generating probability measure $\mu$ with finite second moment. Suppose that $\Gamma$ acts on $Y$ via a homomorphism $\rho: \Gamma \rightarrow \isom{Y}$ satisfying $l_{\mu}(\rho)=0$. Then either $\rho(\Gamma)$ fixes a point in $\partial Y$, or it preserves a flat subspace of $Y$.
\end{thm}

If $\rho(\Gamma)$ has a fixed point in $Y$, then $l_\mu(\rho)=0$ for every probability measure $\mu$ on $\Gamma$ since the orbit is bounded. The following two results relate the vanishing drift to the existence of almost fixed points.

\begin{prop}
\label{lem:strongly_fixed}
 Let $\Gamma$ be a finitely generated group with
 fixed-point property for all finite-dimensional Euclidean spaces, and let $Y$ be a complete
 $\cat{0}$ space with bounded geometry. Suppose that $\rho: \Gamma \rightarrow \isom{Y}$ is a homomorphism such that $$l_{\mu}(\rho)=0$$ for some symmetric generating probability  measure $\mu$ on $\Gamma$ with finite second moment. Then $\rho(\Gamma)$ has almost fixed points. 
 \end{prop}

\begin{proof}
Recall that a CAT(0) space of bounded geometry is always proper. By Theorem \ref{thm:IK23}, either $\rho(\Gamma)$ preserves a flat subspace of $Y$, in which case the fixed-point property for Euclidean spaces yields a global fixed point in $Y$, or $\rho(\Gamma)$ fixes a point in $\partial Y$. It remains to consider the latter case.

Suppose for contradiction that $\rho(\Gamma)$ does not admit almost fixed points:
\begin{equation*}
 \inf_{x \in Y} \max_{s \in S} d(x, \rho(s)x)\geq \delta>0. 
\end{equation*}
Note that $\Gamma$ does not admit any nontrivial homomorphisms into $\R$ according to our assumption. Then we can apply Lemma~\ref{lem:bounded geometry} to obtain a complete $\cat{0}$ space $Z$ with bounded geometry and a homomorphism 
$\rho'\colon \Gamma \rightarrow \isom{Z}$ such that
$\rho'(\Gamma)$ does not fix a point in $Z \cup \partial Z$.

On the other hand, by Lemma~\ref{lem:bounded geometry} we have $l_{\mu}(\rho')=0$. Applying Theorem \ref{thm:IK23} to $\rho'$ it follows that $\rho'(\Gamma)$ has a fixed point in $Z \cup \partial Z$, which yields a contradiction.
\end{proof}

In the finite-dimensional case, the following analog holds. 
\begin{prop}
\label{lem:strongly_fixed_dim}
 Let $\Gamma$ be a finitely generated group with the
 fixed-point property for $\R^n$, and $Y$ a complete $n$-dimensional $\cat{0}$ space. 
 Suppose that $\rho: \Gamma \rightarrow \isom{Y}$ is a homomorphism
 and that there exists a symmetric generating probability measure $\mu$
 with finite second moment such that the drift $l_{\mu}(\rho)$ of $\rho(\Gamma)$ with respect to $\mu$ is equal to zero. 
 Then $\rho(\Gamma)$ has almost fixed points, namely $\delta(\rho)=0$. 
 \end{prop}

\begin{proof}
 Assume first that $\rho(\Gamma)$ fixes a point in $\partial Y$ and does not admit almost fixed points:
\begin{equation*}
 \inf_{x \in Y} \max_{s \in S} d(x, \rho(s)x)= \delta>0. 
\end{equation*}
Note that $\Gamma$ does not admit any nontrivial homomorphisms into $\R$ according to our assumption. 
Then we can apply Lemma~\ref{lem:splitting} to see that there exists a $\cat{0}$ space $Z$ with $\dim Z < \dim Y$ and a homomorphism 
$\rho'\colon \Gamma \rightarrow \isom{Z}$ such that
$\rho'(\Gamma)$ does not fix a point in $\partial Z$.  
We also have 
\begin{equation}
\label{eq:translation_bound}
 \inf_{z \in Z} \max_{s \in S} d(z, \rho'(s)z)\geq \delta >0. 
\tag{$\star \star$}
\end{equation}
Using an argument similar to that of Lemma~\ref{lem:bounded geometry}, it is not hard to see that $l_{\mu}(\rho')=0$. 

 Since $Z$ is finite-dimensional and $\rho'(\Gamma)$ does not fix a
 point in $\partial Z$, there exists a $\rho'$-equivariant
 $\mu$-harmonic map $f\colon \Gamma \rightarrow Z$. 
 Then, by \cite[Proposition 11]{izeki-karlsson}, the convex hull $F$ of
 $f(\Gamma)$ is isometric to $\R^k$ for some $k\leq \dim Z \leq n$.
 Since $F$ is invariant under the action of  $\rho'(\Gamma)$ and $\rho'$ satisfies (\ref{eq:translation_bound}),
 $\rho'(\Gamma)$ does not fix a point in $F$, which contradicts the assumption that $\Gamma$ has the fixed-point property
 for $\R^n$. Therefore, $\rho(\Gamma)$ must have almost fixed points. 

 Now suppose that $\rho(\Gamma)$ does not fix a point in $\partial Y$. 
 Then, since $l_{\mu}(\rho)=0$, again by \cite[Proposition 11]{izeki-karlsson}, the convex hull $F$ of $f(\Gamma)$ is isometric to $\R^k$ for some $k\leq n=\dim Y$. 
 However, since $\Gamma$ has the fixed-point property for $\R^n$, $\rho(\Gamma)$ must fix a point in $F$; this means that $f$ is a constant map and $f(\Gamma)$ is actually a point in $Y$.  In particular, $\delta(\rho)=0$. 
\end{proof}

\section{Torsion groups acting on CAT(0) spaces}
\label{sec:torsion}
\subsection{Overview}
The question of to what extent finitely generated torsion groups can act without fixed point is an old one, as discussed in subsection \ref{ssec:mainresults}. The folklore conjecture has been that, under some conditions, this should not be possible. Recently, Norin, Osajda, and Przytycki made the following precise conjecture: 

\begin{conjecture} \label{conj:nop} \emph{\cite[Conjecture 1.5]{norin-osajda-przytycki}} 
Every finitely generated group acting without a global fixed point on a finite-dimensional $\cat{0}$ complex $Y$ contains an element of infinite order.
\end{conjecture}

See also \cite{haettel-osajda} for more information and a larger context.
In dimension 1 it was proved by Serre in \cite[section 6.5]{serre} and in dimension 2 by Norin-Osajda-Przytycki in \cite{norin-osajda-przytycki} (with a few very mild extra technical assumptions). 

The results of the present paper settle substantial parts of this conjecture:
\begin{itemize}
\item torsion subgroups of a CAT(0) group (Theorem \ref{cor:Swenson-intro} - Swenson's question)
    \item bounded exponent (Theorem \ref{thm: torsion-intro})
    \item amenable groups (Corollary \ref{cor:C})
    \item when $Y$ is a visibility space (Theorem \ref{thm:visibility_fixed_pt})
    \item when $Y$ has compact Tits boundary (Corollary \ref{cor:Titscompacttorsion})
    \item when a geometric Berger-Wang identity holds (Corollary \ref{cor:BergerWang})
\end{itemize}
None of these assumes that $Y$ is a complex, on the other hand, the last three entries assume that the space is geodesically complete and of bounded geometry. The conjecture was established in \cite{izeki-karlsson} in the class of groups with the Liouville property which is a subset of amenable groups.

\subsection{A key lemma} In this subsection, we establish the following lemma, which is essential in the proofs of Theorem \ref{cor:Swenson-intro} and \ref{thm: torsion-intro}.

Let $\Gamma$ be a finitely generated group for which there exists $B<\infty$ such
that  $\mathrm{ord}(g)\leq B$ for every $g\in \Gamma $.
The key ingredient is the following block-swap argument:

\begin{lem}[Bounded-order block-swap lemma]\label{lem:block-swap}
Let a countable group $G$ act by isometries on a complete CAT(0) space $Y$.  Assume
that there exists $B<\infty$ such that
\[
  \mathrm{ord}(g)\leq B\qquad\text{for every }g\in G.
\]
Then

\[
 l_{\mu}(G)=0
\]
for every finitely supported probability measure $\mu$ on $G$.
\end{lem}

\begin{proof}
Let $X_1,X_2,\ldots$ be independent $G$-valued random variables with common
law $\mu$, and set $Z_n=X_1\cdots X_n$.  By Kingman's subadditive ergodic theorem, see for example 
\cite{karlsson-margulis},
\[
  \frac{1}{n}d(y,Z_ny)\longrightarrow \ell:=  l_{\mu}(G) 
\]
almost surely.  Since $\mu$ has finite support, there exists $C<\infty$ such
that $d(y,Z_ny)\leq Cn$; thus the convergence also holds in $L^1$ and in
probability.

Suppose, seeking a contradiction, that $\ell>0$.  For each $m$, let $U_m$
and $V_m$ be independent random variables, both with law $\mu^{*m}$.  Then
\[
  U_mV_m\sim\mu^{*2m},
  \qquad
  V_mU_m\sim\mu^{*2m}.
\]
The second identity does not use commutativity: the ordered pairs
$(U_m,V_m)$ and $(V_m,U_m)$ have the same distribution.  Consequently,
\begin{equation}\label{eq:four-convergence}
\begin{split}
 \frac{1}{m}d(y,U_my)&\longrightarrow\ell,\qquad
 \frac{1}{m}d(y,V_my)\longrightarrow\ell,\\
 \frac{1}{m}d(y,U_mV_my)&\longrightarrow2\ell,\qquad
 \frac{1}{m}d(y,V_mU_my)\longrightarrow2\ell
\end{split}
\end{equation}
in probability.  More explicitly, for every $\eta>0$, the probability that
all four displayed quantities differ from their respective limits by less
than $\eta$ tends to $1$; this follows from the four individual convergences
and the union bound.

Choose integers $m_k\to\infty$ such that the event on which all four
quantities in \eqref{eq:four-convergence} differ from their respective limits
by less than $1/k$ has positive probability.  We may therefore choose
deterministic elements $u_k,v_k\in G$ satisfying
\begin{equation}\label{eq:four-asymptotics}
\begin{aligned}
 d(y,u_ky)&=\ell m_k+o(m_k),&
 d(y,v_ky)&=\ell m_k+o(m_k),\\
 d(y,u_kv_ky)&=2\ell m_k+o(m_k),&
 d(y,v_ku_ky)&=2\ell m_k+o(m_k).
\end{aligned}
\end{equation}
Indeed, the normalized errors are at most $1/k$, so the corresponding
unnormalized errors are $o(m_k)$.

Put $w_k=u_kv_k$ and $q_k=\mathrm{ord}(w_k)$.  By hypothesis,
$q_k\leq B$.  Moreover, \eqref{eq:four-asymptotics} implies that $w_k\neq e$
for all sufficiently large $k$.  After passing to a subsequence and
relabeling, there is a fixed $q\in\{2,\ldots,B\}$ such that
\[
  q_k=q\qquad\text{for every }k.
\]

For $0\leq r\leq q-1$, define
\[
  p^{(k)}_{2r}=w_k^ry,
  \qquad
  p^{(k)}_{2r+1}=w_k^ru_ky,
\]
with indices understood modulo $2q$.  Since $w_k^q=e$, this is a closed
$2q$-gon.  Isometric invariance and \eqref{eq:four-asymptotics} give
\begin{align}
 d\bigl(p^{(k)}_{2r},p^{(k)}_{2r+1}\bigr)
 &=\ell m_k+o(m_k),\label{eq:edge1}\\
 d\bigl(p^{(k)}_{2r+1},p^{(k)}_{2r+2}\bigr)
 &=\ell m_k+o(m_k),\label{eq:edge2}\\
 d\bigl(p^{(k)}_{2r},p^{(k)}_{2r+2}\bigr)
 &=2\ell m_k+o(m_k),\label{eq:diag1}\\
 d\bigl(p^{(k)}_{2r+1},p^{(k)}_{2r+3}\bigr)
 &=2\ell m_k+o(m_k).\label{eq:diag2}
\end{align}
For example, the last equality follows from
\[
\begin{split}
 d\bigl(w_k^ru_ky,w_k^{r+1}u_ky\bigr)
 &=d(u_ky,w_ku_ky)\\
 &=d(y,u_k^{-1}w_ku_ky)
 =d(y,v_ku_ky).
\end{split}
\]
The formulas remain valid at the closing vertices because all indices are
taken modulo $2q$ and $w_k^q=e$.

Fix a nonprincipal ultrafilter $\omega$ and form the pointed ultralimit
\[
  (Y_\omega,\bar y)
  =\omega\text{-}\lim (Y,m_k^{-1}d,y).
\]
which again is a CAT(0) space. 
Since $q$ is fixed and all side lengths are
$O(m_k)$, each sequence $(p_i^{(k)})_k$ defines a point
$\bar p_i\in Y_\omega$.  Equations \eqref{eq:edge1}--\eqref{eq:diag2} imply
\[
  d(\bar p_i,\bar p_{i+1})=\ell>0
\]
and
\[
  d(\bar p_{i-1},\bar p_{i+1})
  =d(\bar p_{i-1},\bar p_i)+d(\bar p_i,\bar p_{i+1})
  =2\ell
\]
for every $i$ modulo $2q$.  Thus, the geodesic sides joining consecutive
vertices form a closed local geodesic of length $2q\ell>0$.  Indeed, equality
in the triangle inequality and uniqueness of geodesics imply that
\[
  [\bar p_{i-1},\bar p_i]\cup[\bar p_i,\bar p_{i+1}]
  =[\bar p_{i-1},\bar p_{i+1}]
\]
for every $i$, including the two indices at the closing junction.
Non-trivial closed geodesics do not exist in CAT(0) spaces.

This contradiction implies that
$\ell=0$.
\end{proof}

\begin{rem}
The bounded-order block-swap lemma in particular shows that for Burnside groups, any isometric action on a CAT(0) space has zero drift. Note that most free Burnside groups are nonamenable and hence, in contrast, their random walks have positive drift in the word metric.
\end{rem}


\subsection{Proof of Swenson's question}
We use the following standard convention: a \emph{CAT(0) group} is a
discrete group admitting a proper and cocompact isometric action on a proper
CAT(0) space.  The action is not assumed a priori to be faithful; properness
implies that its kernel is finite. An isometric action of $G$ on $Y$  is \emph{proper} (or
\emph{discrete}
) if, for every $y\in Y$ and every bounded subset
$A\subseteq Y$, the set
\[
  \{g\in G:gy\in A\}
\]
is finite.

\begin{thm}\label{thm:main} \emph{(Theorem \ref{cor:Swenson-intro} ) }
Every torsion subgroup of a CAT(0) group is finite.
\end{thm}

\begin{proof}
    Let $\Delta$ be a $\cat{0}$ group acting properly and cocompactly by isometries
on a proper CAT(0) space $Y$. As explained in Remark \ref{rem:boundedgeometry} it is a standard fact that $Y$ then necessarily has bounded geometry.

It is a simple well-known consequence of the Cartan-Bruhat-Tits fixed-point theorem
that the number of conjugacy classes of finite subgroups in CAT(0) groups is finite, \cite[Corollary II.2.8]{bridson-haefliger}. Therefore there is an upper bound $C$ on the order of all finite subgroups. In particular, every torsion element in $\Delta$ has order at most $C$. Therefore we can find a number $Q$, an \emph{exponent}, such as $g^Q=1$ for all $g \in \Gamma$.

Take a torsion subgroup $H$ of $\Delta$ generated by a finite set $S$. Thanks to Lemma \ref{lem:block-swap} and Proposition \ref{lem:strongly_fixed} we have the existence of almost fixed points $x_n$:
$$
d(sx_n,x_n)<1/n
$$
for every $s\in S$. Take a compact set $K$ such that $\Delta K=Y$. Transport $x_n$ back to $K$ with $g_n\in\Delta$, so that $y_n:=g_nx_n\in K$. Then we have
$$
d(g_nsg_n^{-1}y_n,y_n)=d(sx_n,x_n)<1/n<1.
$$
By the properness of the action, the set $\{ g_nsg_n^{-1}: n>0\}$ is finite in $\Delta$. since $S$ is finite we can select a subsequence of $n_k$ such that $t_s:=g_{n_k}sg_{n_k}^{-1}$ are independent of $n_k$. Refine this subsequence one more time so that, thanks to the compactness of $K$, $g_{n_k}x_{n_k}$ converges to a point $y\in K$. Now we have for an index $m$ belonging to the selected subsequence that
$$
d(sg_m^{-1}y,g_m^{-1}y)=d(t_sy,y)=\lim_k d(t_sy_{n_k},y_{n_k})=0
$$
and we have the desired global fixed point for $H$. The orbit of $H$ is hence bounded and by properness $H$ must be finite.

Thus we have proved that any finitely generated torsion subgroup of $\Delta$ is finite. 

In case $H$ is not finitely generated, suppose that $H$ is infinite and take $C+1$ distinct elements in $H$. By the above argument, these generate a finite subgroup of cardinality at least $C+1$, which is a contradiction to the upper bound on finite subgroups. Hence, the subgroup $H$ is finite.
\end{proof}

\subsection {Proof of Theorem \ref{thm: torsion-intro}}
In this subsection, we prove Theorem \ref{thm: torsion-intro}, which is for finitely generated torsion groups with finite exponents.
\begin{thm} \emph{(Theorem \ref{thm: torsion-intro})}
    Let $\Gamma$ be a finitely generated group for which there exists an integer $N$ such that $g^N=e$ for all $g\in \Gamma$. Then we have
     \begin{itemize}
 \setlength{\itemsep}{2pt}
     \item[{\rm (1)}]  Whenever $\Gamma$ acts properly on a complete $\cat{0}$ space $Y$ of bounded geometry, $\Gamma$ must be a finite group.
     \item[{\rm (2)}] Whenever $\Gamma$ acts on a complete finite-dimensional $\cat{0}$ space $Y$, there is a global fixed point in $Y$.
 \end{itemize}
\end{thm}

We split the proof into two parts.
\subsubsection{Proper action and bounded geometry} \label{ssec:BGT}
Now we prove the first part of Theorem \ref{thm: torsion-intro} in which the isometric action is proper and the CAT(0) space is of bounded geometry. The discussion also serves as a comparison between the use of Theorem \ref{thmD}, which requires geodesic completeness but makes no proper action assumption, with that of applying a generalized Margulis lemma.

\begin{thm}[Part (1) of Theorem \ref{thm: torsion-intro}]\label{thm:proper-torsion}
    Let $G$ be a finitely generated group acting properly on a complete $\cat{0}$ space of bounded geometry. Assume that $G$ is torsion with a finite bound on the order of all elements. Then $G$ is finite.
\end{thm}

We use the following generalized Margulis lemma. 

\begin{thm}[Breuillard-Green-Tao]\label{thm:BGT}
Let $K>1$, let $Y$ be a metric space in which every ball of radius $4$ can be
covered by at most $K$ balls of radius $1$, and let
$\Gamma\leq\mathrm{Isom}(Y)$ act discretely.  There exists
$\varepsilon=\varepsilon(K)>0$ such that, for every $y\in Y$, the subgroup
\[
  \Gamma_{\varepsilon}(y)
  :=\bigl\langle g\in\Gamma:d(y,gy)<\varepsilon\bigr\rangle
\]
is virtually nilpotent.
\end{thm}

This is \cite[Corollary~11.17]{breuillard-green-tao}; see also the formulation
\cite[Theorem~5.2]{ji-wu_tits_alt}.  Our definition of bounded geometry supplies the
covering hypothesis.  Indeed, put
\[
  K_Y:=\max\{2,n_{4,1}(Y)\}.
\]
A greedy construction of a maximal $1$-separated subset of any ball of
radius $4$ terminates after at most $n_{4,1}(Y)$ steps.  Maximality says that
the open radius-$1$ balls centered at the selected points cover the original
ball. Thus, $K_Y>1$ is a packing constant in the precise sense of
Theorem~\ref{thm:BGT}.

\begin{proof}[Proof of Theorem~\ref{thm:proper-torsion}]
Let $\rho:G\to\mathrm{Isom}(Y)$ denote the given action and let $S=S^{-1}$ be a
finite generating set of $G$.  As before, Lemma \ref{lem:block-swap} and Proposition \ref{lem:strongly_fixed} gives
\[
  \inf_{y\in Y}\max_{s\in S}d(y,\rho(s)y)=0.
\]
Since $Y$ has bounded geometry, the discussion following
Theorem~\ref{thm:BGT} supplies a packing constant $K_Y$, and hence a Margulis
constant $\varepsilon(K_Y)>0$.  Choose $y\in Y$ such that
\[
  \max_{s\in S}d(y,\rho(s)y)<\varepsilon(K_Y).
\]
Put $\overline G:=\rho(G)$.  The action of $\overline G$ is discrete: for a
bounded set $A\subseteq Y$, the set
\[
  \{\bar g\in\overline G:\bar g y\in A\}
\]
is the image of a subset of the finite set
$\{g\in G:\rho(g)y\in A\}$.  Moreover,
\[
  \rho(S)\subseteq
  \{\bar g\in\overline G:d(y,\bar g y)<\varepsilon(K_Y)\}.
\]
Since $\rho(S)$ generates $\overline G$, it follows that
\[
  \overline G=\overline G_{\varepsilon(K_Y)}(y).
\]
Theorem~\ref{thm:BGT} therefore implies that $\overline G$ is virtually
nilpotent.  The group $\overline G$ is finitely generated and torsion, being
a quotient of $G$. 
It is a standard fact that every finitely generated virtually nilpotent torsion group is finite.
Therefore $\overline G$ is finite.

Finally, discreteness of the original action gives
\[
  \ker\rho\subseteq\{g\in G:\rho(g)y\in\{y\}\},
\]
and the set on the right is finite.  Thus $\ker\rho$ is finite.  The exact
sequence
\[
  1\longrightarrow\ker\rho\longrightarrow G
  \longrightarrow\overline G\longrightarrow 1
\]
now shows that $G$ is finite.
\end{proof} 

\subsubsection{The finite dimensional case without proper action and bounded geometry} We complete the proof of the second part of Theorem \ref{thm: torsion-intro}, in which the isometric action is not required to be proper and the space is not required to have bounded geometry.

\begin{thm}[Part (2) of Theorem \ref{thm: torsion-intro}]\label{thm:finite-dim}
    Let $\Gamma$ be a finitely generated group acting on a complete finite-dimensional $\cat{0}$ space. Assume that $\Gamma$ is torsion with a finite bound on the order of all elements. Then there is a global fixed-point for $\Gamma$.
\end{thm}

Note that if we have a fixed-point free action of a finitely generated
group $\Gamma$ on a finite-dimensional complete $\cat{0}$ space, then we
can always find a finite-dimensional complete $\cat{0}$ space $Y$ and a
homomorphism $\rho\colon \Gamma \rightarrow \isom{Y}$ satisfying
\begin{equation*}
 \inf_{q \in Y} \max_{s \in S} d(q, \rho(s) q) >0
\end{equation*}
by the elegant argument in \cite[Proposition 3.1]{bourdon-fix}, see also \cite{izeki-karlsson}. 
In particular, $\rho(\Gamma)$ has no fixed points in $Y$. 
Together with Lemma~\ref{lem:splitting}, we get the following. 

\begin{prop}
 \label{prop:no_fixed_point_in_bdry}
Let $\Gamma$ be a finitely generated group that admits no nontrivial
 homomorphism into $\R$. Suppose that $\Gamma$ acts on 
 a finite-dimensional complete $\cat{0}$ space $Y$ via
 a homomorphism $\rho:\Gamma \rightarrow \isom{Y}$, and that $\rho(\Gamma)$ has no fixed point in $Y$. 
 Then there exists a complete $\cat{0}$ space $Z$, with 
 $\dim Z \leq \dim Y$, on which
 $\Gamma$ acts without fixed points in $Z \cup \partial Z$. 
\end{prop}

\begin{proof}[Proof of Theorem~\ref{thm:finite-dim}]
Assume that $Y$ is a finite dimensional CAT(0) space and that $\Gamma$ has no fixed point in $Y$. Proposition
\ref{prop:no_fixed_point_in_bdry} then provides another finite dimensional $\cat{0}$ space $Z$ with no fixed point in $Z \cup\partial Z$. Take a symmetric probability measure supported on a generating set. Lemma \ref{lem:block-swap} shows that the drift is zero. Theorem \ref{thm:IK23} now gives a fixed point in $\partial Z$ or an invariant flat. In the latter case, the theorems of Schur and Cartan then imply a fixed point in $Z$. This is a contradiction, and hence $\Gamma$ must fix a point in $Y$.
\end{proof}


\section{An alternative for amenable groups acting on $\cat{0}$ spaces}
\label{sec:an-altern-amen}
In the case of amenable groups, we have the following result.

 \begin{thm} \emph{(Theorem \ref{thm:amenable-intro}, main part)} \label{thm:dimensionalternative}
  Let $\Gamma$ be a finitely generated amenable group and $n$ a positive
  integer. Then either
 \vspace{-0.1cm}
 \begin{itemize}
 \setlength{\itemsep}{2pt}
  \item[{\rm (1)}]  every isometric action of $\Gamma$ on an $n$-dimensional complete
  $\cat{0}$ space has a global fixed point, or
  \item[{\rm (2)}] $\Gamma$ admits an isometric action on $\R^n$ without a global fixed point. 
 \end{itemize}
 \end{thm}
 \begin{proof}
Suppose $(1)$ does not hold. That is, there exists an $n$-dimensional complete
  $\cat{0}$ space $Y$ such that an isometric action of $\Gamma$ on $Y$ does not have any fixed points. 
  
If $\Gamma$ admits a nontrivial homomorphism into
$\R$, then $\Gamma$ acts on $\R$ without a global fixed
point. This clearly implies $(2)$. 

If $\Gamma$ admits no nontrivial homomorphism into
$\R$, it follows from Proposition \ref{prop:no_fixed_point_in_bdry} that there exists a complete CAT(0) space $Z$ with $\dim Z \leq n=\dim Y$, on which
 $\Gamma$ acts without fixed points in $Z \cup \partial Z$. Since $\Gamma$ is amenable, according to \cite[Main Theorem]{AB98} (or \cite[Theorem 1.7]{caprace-lytchak}), the action of
$\rho(\Gamma)$ on $Z$ must leave a nontrivial flat subspace invariant, and this action has no fixed point in the flat. This also implies $(2)$. 
 \end{proof}

Farb \cite{farb} and Bridson \cite{bridson} investigated the largest dimension $\mathrm{FixDim}_{\cat{0}}(\Gamma)$ of complete CAT(0) spaces for which $\Gamma$-actions always have a global fixed point. Similarly, define $\mathrm{FixDim}_{\mathrm{Eucl}}(\Gamma )$ to be the largest integer $n$ for which $\Gamma $ has the fixed-point property for $\R^n$.
If no fixed-point-free action exists, we define this dimension to be infinite. For example, $\mathrm{FixDim}_{\cat{0}}(G)=\infty$ for any finite group $G$. 

Now we are ready to prove the second part of Theorem \ref{thm:amenable-intro}.

\begin{cor}[Theorem \ref{thm:amenable-intro}, equal dimension]
For any finitely generated amenable group $\Gamma$,
 $$
\mathrm{FixDim}_{\cat{0}}(\Gamma)=\mathrm{FixDim}_{\mathrm{Eucl}}(\Gamma ).
$$   
\end{cor}

\begin{proof}
The inequality, $\mathrm{FixDim}_{\cat{0}}(\Gamma) \leq \mathrm{FixDim}_{\mathrm{Eucl}}(\Gamma )$, is trivial since Euclidean spaces are special cases of complete CAT(0) spaces.

For the other direction, if $\mathrm{FixDim}_{\cat{0}}(\Gamma)=\infty$, then we are done by the above trivial inequality. Now we suppose that $\mathrm{FixDim}_{\cat{0}}(\Gamma)=n<\infty$. Then there exists a fixed-point free action by $\Gamma$ on an $(n+1)$-dimensional complete CAT(0) space. This means that the first alternative in Theorem \ref{thm:dimensionalternative} is ruled out for the dimension $n+1$ and thus by the second alternative, $\mathrm{FixDim}_{\mathrm{Eucl}}(\Gamma ) \leq n = \mathrm{FixDim}_{\cat{0}}(\Gamma)$. 
\end{proof}

Breuillard pointed out to us that for amenable groups, $\mathrm{FixDim}_{\mathrm{Eucl}}(\Gamma )<\infty $ has an algebraic characterization.
Let
$$
vb_1(\Gamma) = \sup_{[\Gamma :H]<\infty}\dim_{\mathbb{Q}} H^1(H;\mathbb{Q})
$$
be the virtual first Betti number.
\begin{prop}
    Let $\Gamma$ be a finitely generated amenable group. It admits an isometric action on a finite-dimensional Euclidean space without a global fixed point if and only if its virtual first Betti number is strictly positive.
\end{prop}
\begin{proof}
    Suppose first that the virtual first Betti number is nonzero. Then there is a finite index subgroup $H$ of $\Gamma$ with a nontrivial homomorphism $\chi :H\rightarrow \mathbb{R}$. This gives a translation action $\rho$ by $H$ without a fixed point on $\mathbb{R}$. This action can be induced to $\Gamma$ on the affine space $\{ f:\Gamma \rightarrow \mathbb{R} : f(hg)=\rho(h)f(g)=f(g)+\chi(h)\}$, which is isomorphic to $\mathbb{R}^{[\Gamma:H]}$.
    If $f$ were fixed by $H$, then $f(h)=f(e)$ for all $h$, but this contradicts $f(h)=f(e)+\chi(h)$ since $\chi$ is nontrivial.
    Thus this provides a fixed-point free action.

For the other direction, assume that $\Gamma$ has a fixed-point-free action, and consider the image $G$ of $\Gamma$ in the isometry group of $\mathbb{R}^n$. It follows that $G$ is finitely generated, amenable, and hence virtually solvable by the Tits alternative theorem since the isometry group is linear. The group $G$ must also be infinite, otherwise there would be a global fixed point. Passing to a finite index subgroup, we may assume that $G$ is solvable. 

In case $G$ is abelian, then by virtue of being infinite and finitely generated, it must have a $\mathbb{Z}$ factor. Now by induction on the derived length, either $G$ has infinite abelianization $G/[G,G]$ in which case we are done, or otherwise we pass to the finite index subgroup $[G,G]$, which hence is again finitely generated, infinite and solvable, but with lower derived length. This completes the induction. The finite index subgroup obtained with nontrivial first Betti number can now be lifted to a finite index subgroup of $\Gamma$ with nontrivial first Betti number.  
\end{proof}

Schur proved, in response to a famous problem raised by Burnside, that finitely generated linear torsion groups must be finite, which implies that any isometric action of such a group on $\R^n$, any finite $n$, has a fixed point, since the isometry groups of Eculidean spaces are linear.

\begin{cor} \label{cor:C}
Finitely generated, amenable, torsion groups cannot act on 
complete finite-dimensional
$\cat{0}$ space, without a global fixed point.
\end{cor}

It is remarkable that infinite such groups exist. Indeed, infinite, finitely generated torsion groups were first discovered by Golod in the 1960s. The first amenable such examples are the Grigorchuk groups, see \cite{nekrashevych} for further examples and references. 

This corollary answers, in the case of amenable groups, an old and well-known question in geometric group theory, see \cite{norin-osajda-przytycki, haettel-osajda} for the history of this problem, see also subsection \ref{ssec:mainresults} and section \ref{sec:torsion}. 

Since linear groups are residually finite, in view of the Tits alternative theorem, we also have the following corollary (see \cite{izeki-karlsson} for details):

\begin{cor} \label{cor:D}
Finitely generated, amenable, virtually simple groups and finitely generated, amenable, just infinite, nonlinear groups cannot act on 
 a complete 
 $\cat{0}$ space of finite dimension, without a global fixed point.
\end{cor}

Examples of such groups can be found within the important class of branch groups \cite{bartholdigrigorchuksunik}, in particular in the work of Grigorchuk. 
Juschenko-Monod \cite{JM13} proved the existence of finitely generated simple amenable groups, confirming a conjecture of Grigorchuk-Medynets.
Further examples appear in the work of Matte Bon and Nekrashevych,
see \cite{nekrashevych}. Corollary \ref{cor:D} significantly extends a result in \cite{izeki-karlsson}.

\section{Scaling ultralimits}
\label{sec:bound-geom-scal}

Let $Y_n$ be a sequence of complete $\cat{0}$ spaces and $p_n \in Y_n$. Recall that for any sequence $\{r_n\}\subset \mathbb{R}$ with $r_n \to \infty$, and any non-principal ultrafilter $\omega$ on $\mathbb{N}$, the ultralimit $$(Y_\infty,p_\infty)=\omega\text{-}\lim (r_nY_n,p_n)$$ is  a complete $\cat{0}$ space. In what follows, we will deal with sequences of $\cat{0}$ spaces with uniformly bounded geometry; we say a family of complete $\cat{0}$ spaces has {\it uniformly bounded geometry} if, for all $r>\varepsilon>0$ and $Y \in \mathcal{Y}$, the $(r,\varepsilon)$-capacity $n_{r,\varepsilon}(Y)$ is uniformly bounded independent of $Y \in \mathcal{Y}$.  We denote the supremum of $n_{r,\varepsilon}(Y)$ by $n_{r,\varepsilon}(\mathcal{Y})$: $$n_{r,\varepsilon}(\mathcal{Y})= \sup_{Y\in \mathcal{Y}} n_{r,\varepsilon}(Y).$$  The following proposition is fundamental and essential in the proofs of  Theorems  \ref{thm:visibility_fixed_pt}, \ref{thm:F}, and \ref{thm:Bboundedgeom}.
\begin{prop}
\label{prop:tits_boundary}
Let $\mathcal{Y}$ be a family of geodesically complete $\cat{0}$ spaces with uniformly bounded geometry. Given $\{Y_n\} \subset \mathcal{Y}$, $p_n \in Y_n$, $\{r_n\}\subset \mathbb{R}$ with $r_n \to \infty$, consider the ultralimit $(Y_\infty,p_\infty)=\omega\text{-}\lim (r_nY_n,p_n)$. Then the following holds:
\begin{itemize}
 \setlength{\itemsep}{2pt} 
 \item[{\rm (1)}] For any $r>\varepsilon>0$,
 $$n_{r,\varepsilon}(\mathcal{Y})\geq n_{r,\varepsilon}(Y_\infty).$$
 In particular, $Y_\infty$ also has bounded geometry.
 \item[{\rm (2)}] The Tits boundary $\partial_T Y_{\infty}$ of $Y_\infty$ is compact.
\end{itemize}
\end{prop}

Before proving Proposition \ref{prop:tits_boundary}, we establish two useful lemmas.
\begin{lem}
\label{lem:capacity}
Suppose $X$ is a geodesically complete $\cat{0}$ space with bounded geometry, and for every $r>\varepsilon>0$, denote by $n_{r,\varepsilon}$ its $(r,\varepsilon)$-capacity. Then for any $a>1$, we have $$n_{ar,a\varepsilon} \geq n_{r,\varepsilon}.$$
\end{lem}
\begin{proof}
For any $p \in X$, let $\{x_1,\cdots,x_N\} \subset B(p,r)$ be an $\varepsilon$-sparse subset. Since $X$ is geodesically complete, every geodesic segment starting from $p$ can be extended (not necessarily uniquely) to a geodesic ray. For every $1 \leq i \leq N$, let  $c_i:[0,\infty) \to X$ be a geodesic ray with $c_i(0)=p$ that passes through $x_i$. Choose $y_i \in c_i$ so that $d(p,y_i)=ad(p,x_i)<ar$. By standard $\cat{0}$ estimates we obtain $$d(y_i,y_j) \geq ad(x_i,x_j) \geq a \varepsilon \; (i\neq j).$$ Thus $\{y_1,\cdots,y_N\}$ is an $a\varepsilon$-sparse subset of $B(p,ar)$. The claim then follows from the arbitrariness of $p$.
\end{proof}

Recall that a family $\mathcal{X}$ of metric spaces is said to be 
 {\it uniformly compact} 
 (\cite[p.~74, 5.39 Definition]{bridson-haefliger}) if 
 \vspace{-0.1cm}
 \begin{itemize} 
  \setlength{\itemsep}{0.1cm}
  \item[(i)] There exists $C>0$ such that $\mathrm{diam}(X)\leq C$ for
  any $X \in \mathcal{X}$, and 
  \item[(ii)] For any $\varepsilon>0$, there exists 
  $N(\varepsilon)\in \N$ such that $N(\varepsilon)$ balls of radius
  $\varepsilon$ cover $X$ for any $X \in \mathcal{X}$; in other words, 
  $n_{2\varepsilon}(X)\leq N(\varepsilon)$ for any $X \in \mathcal{X}$. 
 \end{itemize}
 Note that, if $\mathcal{Y}$ is a family of complete $\cat{0}$ spaces with uniformly bounded geometry, then, for any fixed $r>0$, $\{\bar B(p,r) \mid Y \in \mathcal{Y},\ p \in Y\}$ is uniformly compact, where $\bar B(p,r)$ denotes the closed ball in $Y$ centered at $p$ with radius $r$.

The second lemma is as follows.
\begin{lem}\emph{(compare \cite[p.~79, 5.52 Exercise]{bridson-haefliger})}
\label{lem:ultralimit-GH}
Let $\{(X_n,p_n)\}$ be a uniformly compact sequence of pointed metric spaces. For any non-principal ultrafilter $\omega$ on $\mathbb{N}$, there exists a subsequence that converges in the Gromov-Hausdorff topology to a pointed metric space $(D,x)$, such that the ultralimit $(X_\infty,p_\infty)=\omega\text{-}\lim (X_n,p_n)$ is isometric to $(D,x)$.
\end{lem}
\begin{proof}
We first show that $X_\infty$ has bounded geometry. For any $\varepsilon>0$, let $\{x_1,\cdots,x_K \}\subset X_\infty$ be an $\varepsilon$-sparse set. By the definition of ultralimits, for some sufficiently large $k \in \mathbb{N}$ we can find points $x_{k,1},\cdots,x_{k,K} \in X_k$ so that
\begin{equation*}
|d_{X_k}(x_{k,i},x_{k,j})-d_{X_\infty}(x_i,x_j)|<\varepsilon/2 \;\; (1\leq i,j \leq K).
\end{equation*}
It follows that $\{x_{k,1},\cdots,x_{k,K}\}$ is an $\varepsilon/2$-sparse set in $X_k$, whence $K \leq N(\varepsilon/4)$. 

Now suppose that the $\varepsilon$-sparse set $\{x_1,\cdots,x_K \}\subset X_\infty$ is chosen to be maximal, then it is $\varepsilon$-dense in $X_\infty$. In this case for every $1 \leq i \leq K$, let $\{x_{n,i}\}$ be a sequence with $x_{n,i} \in X_n$ whose ultralimit represents $x_i$. We denote $x_i=(x_{n,i})_{X_\infty}$. 

We select a sequence $\{q_n\}$ with $q_n \in X_n$ as follows. If $X_n$ is covered by $\bigcup\limits_{i=1}^K B(x_{n,i},2\varepsilon)$, then set $q_n=x_{n,1}$; otherwise choose any $q_n \in X_n \setminus  \bigcup\limits_{i=1}^K B(x_{n,i},2\varepsilon)$. We claim that $\omega\{n:q_n=x_{n,1}\}=1$, otherwise we would have $d_{X_\infty}(x_i,(q_n)_{X_\infty})=\omega\text{-}\lim d_{X_n}(x_{n,i},q_n)\geq 2\varepsilon$ for all $1 \leq i \leq K$, which contradicts the maximality of $\{x_1,\cdots,x_K \}$ in $X_\infty$. Therefore, for $\omega$-almost every $n$, the set $\{x_{n,1},\cdots,x_{n,K}\}$ is $2\varepsilon$-dense.

Therefore, the Gromov-Hausdorff distance $D_H(X_k,X_\infty) \leq 2\varepsilon$ for infinitely many $k$. Letting $\varepsilon \to 0$ and applying Gromov's precompactness criterion(\cite[p.~77, 5.45 Theorem]{bridson-haefliger}), we obtain a subsequence of $\{(X_n,p_n)\}$ that converges in the Gromov-Hausdorff topology to a pointed metric space $(D,x)$, which is isometric to $(X_\infty,p_\infty)$.
\end{proof}

Now we are ready to prove Proposition~\ref{prop:tits_boundary}.
\begin{proof}[Proof of Proposition~\ref{prop:tits_boundary}]
$(1)$. For any $q \in Y_\infty$, let $\{x_1,\cdots,x_N\}$ be a finite $\varepsilon$-sparse subset of $B(q,r)$. Then $r':=\max_{1\leq i \leq N} d_{Y_\infty}(q,x_i)<r$. For every $n \in \mathbb{N}^*$ and $1 \leq i \leq N$, choose $q_n, x_{n,i} \in r_nY_n$ so that $q= \omega\text{-}\lim q_n$ and $x_i=\omega\text{-}\lim x_{n,i}$. We have $\omega\text{-}\lim d_{r_nY_n}(q_n,x_{n,i})\leq r'$ and $\omega\text{-}\lim d_{r_nY_n}(x_{n,i},x_{n,j})\geq \varepsilon$ for all $i \neq j$. By the definition of ultralimit, for some sufficiently large $n$, we have
\begin{itemize}
 \setlength{\itemsep}{2pt}
\item[{\rm (i)}] $r_n>1$;
 \item[{\rm (ii)}] $d_{r_nY_n}(q_n,x_{n,i}) < \frac{r+r'}{2}$;
 \item[{\rm (iii)}] $d_{r_nY_n}(x_{n,i},x_{n,j})\geq \frac{r'+r}{2r} \varepsilon$.
\end{itemize}
Since each scaling space $r_nY_n$ is also geodesically complete, we can extend the unit speed geodesic segment emanating from $q_n$ and terminating at $x_{n,i}$ so that the new endpoint $x'_{n,i}$ satisfies $d_{r_nY_n}(q_n,x'_{n,i})=\frac{2r}{r+r'} d_{r_nY_n}(q_n,x_{n,i})<r$. We have 
\vspace{-0.1cm}
\begin{equation*}
d_{r_nY_n}(x'_{n,i},x'_{n,j}) \geq \dfrac{2r}{r+r'} d_{r_nY_n}(x_{n,i},x_{n,j})\geq \varepsilon \ \textrm{ for all } i \neq j.
\end{equation*}
Therefore $\{x'_{n,1},\cdots,x'_{n,N}\}$ is an $\varepsilon$-sparse subset of $B(q_n,r) \subset r_n Y_n$. Via the natural identification $i_n :r_nY_n \to Y_n$, we get immediately that $\{i_n(x'_{n,1}),\cdots,i_n(x'_{n,N})\}$ is a $\frac{\varepsilon}{r_n}$-sparse subset of $B(q_n,\frac{r}{r_n})\subset Y_n$, thus $$N \leq n_{\frac{r}{r_n},\frac{\varepsilon}{r_n}}(Y_n) \leq n_{r,\varepsilon}(Y_n) \leq n_{r,\varepsilon}(\mathcal{Y}),$$ where the second inequality follows from Lemma \ref{lem:capacity}. This finishes the proof of Part $(1)$.
\vspace{0.1cm}

$(2)$. First note that, since $r_n\to \infty$, 
the space $(Y_{\infty},p)=\omega\text{-}\lim (r_nY_n,p_n)$ is isometric to 
 $\omega\text{-}\lim (r_n \bar B(p_n,r),p_n)$ for any $r>0$. 
Fix $r>0$, and consider the ultralimit $\omega\text{-}\lim (\bar B(p_n,r),p_n)$.
Since $\mathcal{Y}$ has uniformly bounded geometry, by Lemma \ref{lem:ultralimit-GH}, $\omega\text{-}\lim (\bar B(p_n,r),p_n)$ is isometric to the Gromov-Hausdorff limit  of a subsequence of $\{(\bar B(p_n,r),p_n)\}$, which we denote by $(D,x)$.  Note that $D$, and hence $\partial D$, is compact. 

 We will show that $\partial D$ bounds $\partial_T Y_{\infty}$. 
 
 Since $Y_n$ is geodesically complete, any geodesic can be extended to a geodesic line. Therefore for every $\xi \in \partial Y_\infty$, the geodesic ray  $c_p^{\xi}$ in $Y_\infty$ with $c_p^\xi(0)=p$, $c_p^\xi(\infty)=\xi$ is a scaling ultralimit of geodesic segments $c_{p_n}^{q_n}$ in $\bar B(p_n,r)$ with endpoints $p_n$ and $q_n \in \partial B(p_n,r)$. We denote $c_p^{\xi}=(c_{p_n}^{q_n})_{Y_\infty}$. On the other hand, the ultralimit (without scaling) of $c_{p_n}^{q_n}$   
 represents a geodesic segment in $D$ starting from $x$ and terminating at $\partial D$, which we denote by $(c_{p_n}^{q_n})_D$. Every geodesic in $Y_{\infty}$ (resp.~$D$) starting from $p$
 (resp.~$x$) and terminating at $\xi \in \partial Y_{\infty}$
 (resp.~terminating at $q\in \partial D$) is obtained in this way. 

 Now take any $\xi, \xi' \in \partial Y_{\infty}$.  
 Then we can find sequences $\{c_{p_n}^{q_n}\}$ and $\{c_{p_n}^{q_n'}\}$ of geodesic segments with $q_n,q'_n \in \partial B(p_n,r)$ so that
 $c_p^{\xi}=(c_{p_n}^{q_n})_{Y_{\infty}}$ and 
 $c_p^{\xi'}=(c_{p_n}^{q_n'})_{Y_{\infty}}$. 
 Note that, because of our scaling by $\{r_n\}$, we have
 \begin{equation*}
  c_p^{\xi}(t) =(c_{p_n}^{q_n})_{Y_{\infty}}(t) =
   (c_{p_n}^{q_n}(t/r_n))_{Y_{\infty}}, \quad
  c_p^{\xi'}(t) = (c_{p_n}^{q_n'})_{Y_{\infty}}(t)=
  (c_{p_n}^{q_n'}(t/r_n))_{Y_{\infty}}. 
 \end{equation*}
 Suppose that $(c_{p_n}^{q_n})_D=(c_{p_n}^{q_n'})_D$ holds. 
 Then we see that 
 $(c_{p_n}^{q_n})_{Y_{\infty}}=(c_{p_n}^{q_n'})_{Y_{\infty}}$ holds.  
 Indeed, for any $t> 0$, we have
\begin{equation*}
\begin{aligned}
  d_{Y_{\infty}}((c_{p_n}^{q_n})_{Y_{\infty}}(t), 
  (c_{p_n}^{q_n'})_{Y_{\infty}}(t))
  &= \omega\text{-}\lim  r_n d_{Y_n}(c_{p_n}^{q_n}(t/r_n),
  c_{p_n}^{q_n'}(t/r_n)) \\
  & \leq \omega\text{-}\lim \dfrac{t}{r} d_{Y_n}(c_{p_n}^{q_n}(r),c_{p_n}^{q_n'}(r)) \\
  & =\dfrac{t}{r} d_{D}((c_{p_n}^{q_n})_D(r), (c_{p_n}^{q_n'})_D(r))=0, 
\end{aligned} 
\end{equation*}
 where the inequality in the second line follows from the convexity of
 the distance function (or simply the $\cat{0}$ property) of $Y$ and the fact that $t/r_n <r$ when $n$ is sufficiently large.   
 Thus we have a well-defined surjection
\begin{equation*}
 \Psi \colon \partial D \rightarrow \partial Y_{\infty}; \ 
 q=(q_n)_D \mapsto \xi=(c_{p_n}^{q_n})_{Y_{\infty}}(\infty). 
\end{equation*}

 Now we take a look at the angular metric $\angle(\cdot, \cdot)$ on
 $\partial Y_{\infty}$, which can be expressed as 
\begin{equation*}
 2\sin \frac{\angle(\xi,\xi')}{2}
 = \lim_{t\to \infty} \frac{1}{t} d_{Y_{\infty}}(c_p^{\xi}(t),c_p^{\xi'}(t))
\end{equation*}
 by \cite[p.~281, 9.8 Proposition (4)]{bridson-haefliger}. 
 For fixed $t>0$, we have 
\begin{equation*}
   \frac{1}{t} d_{Y_{\infty}}(c_p^{\xi}(t),c_p^{\xi'}(t))
  = \frac{1}{t} \omega\text{-}\lim
    r_nd_{Y_n}(c_{p_n}^{q_n}(t/r_n), c_{p_n}^{q_n'}(t/r_n))
  = \omega\text{-}{\lim} \frac{r_n}{t} d_{Y_n}(c_{p_n}^{q_n}(t/r_n),
  c_{p_n}^{q_n'}(t/r_n))
\end{equation*}
by definition. 
Take $n$ large so that $rr_n/t \geq 1$ holds. 
Then again by the convexity of the distance function,  we have
\begin{equation*}
 \frac{rr_n}{t} d_{Y_n}(c_{p_n}^{q_n}(t/r_n),  c_{p_n}^{q_n'}(t/r_n))
 \leq d_{Y_n}(c_{p_n}^{q_n}(r), c_{p_n}^{q_n'}(r)). 
\end{equation*}
 Thus,  for any $t$, by
 taking the ultralimit of both sides, we obtain 
  \begin{equation*}
   \frac{1}{t} d_{Y_{\infty}}(c_p^{\xi}(t),c_p^{\xi'}(t)) \leq
    \frac{1}{r} d_D((q_n)_D, (q_n')_D), 
  \end{equation*}
 and $q = (q_n)_D$ and  $q'=(q_n')_D$ are points in $\partial D$. 
 Note that, by the definition of $\Psi$, we have $\Psi(q)=\xi$ and
 $\Psi(q')=\xi'$. 
 Thus we obtain 
\begin{equation*}
 2\sin \frac{\angle(\Psi(q),\Psi(q'))}{2} \leq \frac{1}{r} d_{D}(q,q'). 
\end{equation*}
 Therefore the surjection 
 $\Psi \colon \partial D \rightarrow  (\partial Y_{\infty},\angle)$ is 
 a continuous map, and we see that
 $(\partial Y_{\infty},\angle)$ is compact. 
 Since $(\partial Y_{\infty},\angle)$ is homeomorphic to
 the Tits boundary $\partial_TY_{\infty}$,  $\partial_T Y_{\infty}$ is
 compact. 
\end{proof}

\section{A Kazhdan-type rigidity principle}
\label{sec:exists-an-almost}

In \cite[Proposition 3.5]{breuillard-fujiwara}, Breuillard and Fujiwara consider the \emph{asymptotic joint displacement} of a finite set $S$ of whose elements act on a metric $X$ by isometries
\begin{equation*}
L_{\infty}(S)= \lim_{n\rightarrow \infty} \max_{\gamma \in S^n} d(x,\gamma x)/n.
\end{equation*}
When $X$ is a finite-dimensional Euclidean space, they deduce that the group generated by $S$ has almost fixed points if and only if $L_{\infty}(S)=0$. An adaptation of the argument in Proposition \ref{lem:strongly_fixed_dim} allows us to generalize this result to all complete $\cat{0}$ of finite dimension.

\begin{prop}
Let $Y$ be a complete finite-dimensional $\cat{0}$ space and $\Gamma$ a finitely generated group acting on $Y$ via a homomorphism $\rho: \Gamma \rightarrow \isom{Y}$. Then $\rho(\Gamma)$ has almost fixed points in $Y$ if and only if $L_{\infty}(S)=0$ for some (and therefore any) finite generating set $S \subset \Gamma$, i.e., $L_{\infty}(S)=0$ if and only if $\delta(\rho)=0$.
\end{prop}
\begin{proof}
It is straightforward to verify $L_\infty(S)\leq \delta(\rho)$, see, for example, \cite[Lemma 1.1]{breuillard-fujiwara}. Thus, $\delta(\rho)=0$ implies $L_{\infty}(S)=0$.

Conversely, suppose that $L_{\infty}(S)=0$. Let $\mu$ be a symmetric probability measure supported on $S$. We have
\begin{eqnarray*}
l_{\mu}(\rho)&=&\lim_{n\to \infty}
 \frac{\int_{\Gamma}d(x,\rho(\gamma)x) d\mu^{*n}(\gamma)}{n} \\
 &\leq& \lim_{n \to \infty} \frac{\int_{\Gamma} \max_{\gamma_0 \in S^n} d(x,\rho(\gamma_0) x) d\mu^{*n}(\gamma)}{n}\\
 &=& L_\infty(S)=0.
\end{eqnarray*}
It then follows from Theorem \ref{thm:IK23} that either $\rho(\Gamma)$ fixes a point in $\partial Y$, or it preserves a flat subspace of $Y$. The latter case reduces to the Euclidean case, which was settled in \cite[Proposition 1.10]{breuillard-fujiwara}.

It remains to consider when $\rho(\Gamma)$ fixes a point in $\partial Y$. By Lemma \ref{lem:ultralimit}, there exists a complete $\cat{0}$ space $Z'=Z_0 \times \mathbb{R}$, on which $\Gamma$ acts via a homomorphism $\rho':\Gamma \rightarrow \isom{Z}$ preserving the splitting. Let $\pi:Z'=Z_1 \times \mathbb{R} \to \mathbb{R}$ be the projection onto the second factor, we claim that $\rho'(\Gamma)$ acts trivially on the $\mathbb{R}$-factor. Otherwise, there exist $p' \in Z'$ and $\gamma \in S^k$ such that $d(\pi(p'),\pi(\rho'(\gamma)p'))>0$. This would imply
\begin{eqnarray*}
L_\infty(S) &\geq& \dfrac{\lim_{n \to \infty} d(x,\rho(\gamma^n)x)}{kn} \\
& \geq & \dfrac{\lim_{n \to \infty} d(p',\rho'(\gamma^n)p')}{kn} \\
& \geq & \lim_{n \to \infty} \dfrac{d(\pi(p'), \pi(\rho'(\gamma^n)p'))}{kn} \\
&=& \dfrac{\lim_{n \to \infty} n d(\pi(p'), \pi(\rho'(\gamma)p'))}{kn} \\
&=& \dfrac{d(\pi(p'), \pi(\rho'(\gamma)p'))}{k}>0,
\end{eqnarray*}
contradicting $L_\infty(S)=0$. Hence $\rho'(\Gamma)$ acts trivially on the $\mathbb{R}$-factor and therefore induces an isometric action on $Z_1$, with $\dim Z_1<\dim Y$. 

Since $\dim Y$ is finite, we can only iterate this procedure finitely many times, and eventually obtain a fixed point. An argument similar to that in the last paragraph of Lemma \ref{lem:splitting} then yields an almost fixed point in $Y$.
\end{proof}

\begin{rem}
In an infinite dimensional Hilbert space, it is possible to have $L_\infty(S)=0$ while $\delta(\rho)>0$, see \cite[section 3.9]{cornulier-tessera-valette} and the discussion in \cite[section 3.2]{breuillard-fujiwara}.
\end{rem}

For $\cat{0}$ spaces with compact Tits boundary, which are considered to be close to finite-dimensional Euclidean spaces in a sense, Propositions \ref{lem:strongly_fixed} and \ref{lem:strongly_fixed_dim} imply the existence of almost fixed points as follows. 

\begin{prop}
\label{prop:bdd_geom_almost_fix}    
 Let $Z$ be a complete $\cat{0}$ space with bounded geometry and compact Tits boundary $\partial_T Z$.  Suppose either
 \begin{itemize}
 \setlength{\itemsep}{2pt}
     \item[{\rm (i)}] $\Gamma$ has the fixed-point property for all finite-dimensional Euclidean spaces, or
     \item[{\rm (ii)}] there exists a positive integer $n$ such that $\Gamma$ has the fixed-point property for $\R^n$ and $\dim Z \leq n$.  
 \end{itemize}
 Then, for any homomorphism $\rho \colon \Gamma \rightarrow \isom{Z}$, $\delta(\rho)=0$ holds. 
\end{prop}

\begin{proof}
 Under either assumption (i) or (ii), $\Gamma$ does not admit a nontrivial homomorphism into $\R$.
 Using this fact, we first show that the drift $l_{\mu}(\rho)$ of $\rho(\Gamma)$ with respect to any symmetric generating probability measure $\mu$ with finite second moment is equal to zero.
 
 Note that there is a natural homomorphism 
$\Phi\colon \isom{Z} \rightarrow \isom{\partial_TZ}$,
 and $\isom{\partial_TZ}$ is compact because of the compactness
 of $\partial_T Z$. (We denote $\Phi(\rho_{\infty}(g))$ by
 $\rho_{\infty}(g)$ in what follows.) 
Note also that $\partial_T Z$ is
 homeomorphic to the usual geometric boundary $\partial Z$,
 since the identity map 
$\partial_T Z\rightarrow \partial Z$ is continuous, 
$\partial_{T} Z$ is compact and $\partial Z$ is 
 Hausdorff. 

Take any symmetric generating probability measure $\mu$ with finite
 second moment on $\Gamma$ and consider a random walk generated by $\mu$. 
First suppose that the drift $l_{\mu}(\rho_{\infty})$ of 
$\rho_{\infty}(\Gamma)$ is positive.  
Then, by \cite{karlsson-margulis} of Karlsson-Margulis, almost every sample walk
 converges to a point in $\partial Z$, and there is an
 equivariant map $\varphi$ from the Poisson boundary 
$(\partial_P\Gamma, \nu)$ to $\partial_T Z$. 
Let $m=\varphi_* \nu$. For almost all random walk paths $g_n$, one has
the sequence of measures 
$\{(\rho_{\infty}(g_n))_*(m) \rightarrow \delta_{\eta}
 \}$ converging to a Dirac measure supported on a point $\eta \in \partial Z$. 
 
Since $\isom{\partial_TZ}$ is compact, we may assume
 that $\rho_{\infty}(g_n)$ converges (up to a subsequence) to some 
$h\in \isom{\partial_TZ}$ (depending on the choice of $\{g_n\}$).
Therefore $h_*m=\delta_{\eta}$ 
and since $h$ is invertible $m$ is also a Dirac measure, say $m=\delta_{\xi}$. Now, 
$m$ is $\mu$-stationary, so 
$$
\delta_{\xi}=\sum_{\gamma} \mu(\gamma) \delta_{\rho_{\infty}(\gamma)\xi}. 
$$
Thus every element in $\mathrm{supp}\mu$ fixes $\xi$ and since the support generates the whole group, $\Gamma$ fixes $\xi$.

Now denote by $c_p^q$ the geodesic starting from $p$ and terminating at $q$. 
The proof of \cite[Theorem 2.1]{karlsson-margulis} asserts that, for almost
 all sample walk $g_n$ of $(\Gamma,\mu)$, 
the sequence of geodesics $\{c_p^{\rho_{\infty}(g_n)p}\}$ converges
 to the geodesic ray $c_p^{\xi}$ starting from $p$ and terminating at
 $\xi\in \partial Z$.  
 Furthermore, since we already assume that the drift $l_{\mu}(\rho_{\infty})>0$,
 we have
\begin{equation*}
 \lim_{n\to\infty} b_{\xi}(\rho_\infty(g_n)p,p)=-\infty
\end{equation*}
 as explained in \cite[Proposition 4.2]{karlsson-linear} or 
 \cite[\S 3]{karlsson-ledrappier}, where $b_{\xi}(\cdot,p)$ is the
 Busemann function associated to the geodesic ray $c_p^{\xi}$ defined as
\begin{equation*}
 b_{\xi}(q,p)= \lim_{t\to \infty}
   d(q,c_p^{\xi}(t))-t.
\end{equation*}
 However, since $\rho_{\infty}(\Gamma)$ fixes $\xi$ and $\Gamma$ does not
 admit any nontrivial homomorphism into $\R$, $\rho_{\infty}(\Gamma)$ leaves
 each horosphere centered at $\xi$
 (which is given as the level set of $b_{\xi}(\cdot,p)$) invariant, 
 and hence $b_{\xi}(\rho(g_n)p,p)$ must be constant.
This leads to a contradiction. Thus, we conclude that the drift $l_{\mu}(\rho_{\infty})$ is equal to $0$. 

 Suppose that $\Gamma$ satisfies the assumption (i). Since $Z$ has bounded geometry, in view of Proposition~\ref{lem:strongly_fixed}, this implies that $\rho(\Gamma)$ has almost fixed points in $Z$, that is $\delta(\rho)=0$. 
 Suppose that $\Gamma$ and $Z$ satisfy the assumption (ii).  
 Then, by Proposition~\ref{lem:strongly_fixed_dim}, we see that $\delta(\rho)=0$. 
 The proof is complete. 
%
\end{proof}

Using Proposition~\ref{prop:tits_boundary} and
Proposition~\ref{prop:bdd_geom_almost_fix} above, we prove Theorem \ref{thmD}. 

\begin{thm}
\label{thm:almost_fix} \emph{(Theorem \ref{thmD})}
 Let $\Gamma$ be a group generated by a finite set $S$ and $\mathcal{Y}$ a family of geodesically complete $\cat{0}$ spaces with uniformly bounded geometry. 
 Suppose either
 \begin{itemize}
 \setlength{\itemsep}{2pt}
     \item[{\rm (i)}] $\Gamma$ has the fixed-point property for all finite-dimensional Euclidean spaces, or
     \item[{\rm (ii)}] there exists a positive integer $n$ such that $\Gamma$ has the fixed-point property for $\R^n$ and $\dim Y \leq n$ for all $Y \in \mathcal{Y}$. 
 \end{itemize}
 Then there is a constant $\kappa (S, \mathcal{Y})>0$ such that, for any $Y\in \mathcal{Y}$ and homomorphism $\rho :\Gamma \rightarrow \mathrm{Isom}(Y)$ providing a fixed-point-free action, the inequality
 $\delta(\rho ) \geq \kappa (S ,\mathcal{Y})$ holds.

\end{thm}


\begin{rem}\label{remark-boundedgeometry} Let $\Gamma$ be a finitely generated group with the fixed-point property for all finite-dimensional Euclidean spaces, and $Y$ a geodesically complete $\cat{0}$ space with bounded geometry.  Taking $\mathcal{Y}$ to be a family consisting of a single $\cat{0}$ space $Y$ in Theorem \ref{thm:almost_fix}, we see that, in particular, the action given by $\rho\colon \Gamma \rightarrow \isom{Y}$ with $\delta(\rho)=0$ admits a fixed point in $Y$.
 Here the assumptions ``$\Gamma$ has fixed-point property for
 $\R^n$'' and ``$Y$ has bounded geometry'' are both essential. 
 Let $Z$ be a $\cat{0}$ space, and consider a warped product 
$Y=\R\times_f Z$.  If we take $f(t)=e^t$, then $Y$ is known to be a
 $\cat{-1}$ space (\cite[Theorem 1.1]{alexander-bishop}), and 
$\R \times \{z\}$ becomes a geodesic line in $Y$ for each $z\in Z$.  
Suppose a group $\Gamma$ acts on $Z$ without fixed point. 
Then the action of $\Gamma$ on $Z$ extends naturally to a fixed-point
free action on $Y$ by letting $\Gamma$ act trivially on $\R$,  
while it strongly fixes a point in $\partial Y$ represented by geodesic
 rays given by the form of $t \mapsto (-t,z)$, $z \in Z$; 
in particular, there are almost fixed points along these geodesic rays.
For example, we can take an infinite, finitely generated amenable torsion group with a fixed point free action on a Hilbert space $Z$. 
If we take $Z$ to be $\R^n$ with a standard $\Z^n$-action by
 translations, the warped product above gives us a hyperbolic space
 $\mathbb{H}^{n+1}$, which has bounded geometry, and a $\Z^n$-action
 with almost fixed points. 
On the other hand, it is easy to see that unless the boundedness of the
geometry of $Z$ is (almost) invariant under scaling (i.e., almost
$n_{\alpha r,\alpha \varepsilon}(Z)=n_{r,\varepsilon}(Z)$ holds for
$\alpha>0$), like $\R^n$, $Y$ cannot have bounded geometry. In the opposite case, Breuillard-Fujiwara observed \cite[Prop 1.10]{breuillard-fujiwara} that for any set of isometries of a Euclidean space that possesses almost fixed-points, there is a common fixed-point.
\end{rem}

\begin{proof}[Proof of Theorem~\ref{thm:almost_fix}]
Suppose first that there exist $Y\in \mathcal{Y}$ and an action $\rho(\Gamma)$ on $Y$ which does not have a fixed point in $Y$, while there are almost fixed points, so $\delta (\rho)=0$. Then
 take $q_n$ so that $$1/n \geq \delta (q_n):=\max_{s\in S}d(q_n,\rho(s)q_n)$$
 for some finite generating set $S$ of $\Gamma$. Then, as in \cite[Lemma 3.2]{bourdon-fix}, 
 we can take $p_n \in Y$ and $\{r_n\}\subset \R$ with $r_n\to \infty$
 so that $(Y_{\infty},p)=\omega\text{-}\lim (r_nY,p_n)$ and 
 $\rho_{\infty}\colon \Gamma \rightarrow \isom{Y_{\infty}}$
  satisfies $$\inf_{q \in Y_{\infty}} \delta_{\infty}(q)>0,$$ where
 $\delta_{\infty}(q)=\max_{s\in S}d(q,\rho_{\infty}(s)q)$. By applying Proposition~\ref{prop:tits_boundary} to the family consisting of a single $\cat{0}$ space $Y$ with bounded geometry, $Y_{\infty}$ has bounded geometry and compact Tits boundary.
 Noting that the ultralimit does not increase the dimension, under either assumption (i) or (ii), we see that this contradicts Proposition~\ref{prop:bdd_geom_almost_fix}.  

In the complementary case, all fixed-point-free actions on $Y\in \mathcal{Y}$ have strictly positive minimal displacement. Then seeking a contradiction, assume that there exist $\{Y_n\} \subset \mathcal{Y}$  and homomorphisms $\rho_n:\Gamma \rightarrow \mathrm{Isom}(Y_n)$ such that $\delta_n:=\delta(\rho_n)>0$ and $\delta_n \searrow 0$. Take $y_n \in Y_n$ such that
\begin{equation}
 \label{eq:almost}
 \max_{s\in S} d_{Y_n}(y_n, \rho_n(s)y_n) \leq 2\delta_n,
 \tag{$\dagger$ }
 \end{equation}
and take an ultralimit $(Y_{\infty},y_{\infty})=\omega\text{-}\lim ((1/\delta_n)Y_n,y_n)$.
By (\ref{eq:almost}), we have a well-defined homomorphism $\rho_{\infty}:\Gamma \rightarrow \textrm{Isom}(Y_{\infty})$ and, by our construction, we see that $\delta (\rho_{\infty})\geq 1$.
By Proposition~\ref{prop:tits_boundary}, $Y_{\infty}$ has bounded geometry and compact Tits boundary $\partial_{T}Y_{\infty}$. Again, under either assumption (i) or (ii), this contradicts Proposition~\ref{prop:bdd_geom_almost_fix}.  

The proof is complete.
\end{proof}

Although the constant $\kappa(S,\mathcal{Y})$ should depend on the group $G$ generated by $S$, it behaves well on the space of marked groups.  In what follows, we denote by $\kappa_S(G,\mathcal{Y})$ the constant $\kappa(S,\mathcal{Y})$ for a group $G$ generated by $S$.  
We will show that if the dimension of $Y \in \mathcal{Y}$ is uniformly bounded, then  $\kappa_S(\Gamma,\mathcal{Y})$ defines a lower semicontinuous function on the space of marked groups as stated in Proposition~\ref{prop:lower_semiconti_dim} below.  

Recall that the space $\mathcal{G}(k)$ of $k$-{\it marked groups} over $S=\{{s_1}^{\pm 1}, \dots, {s_k}^{\pm 1}\}$ consists of groups obtained as quotients of the free group $F_k$ of rank $k$; $\mathcal{G}(k)$ is identified with the set of normal subgroups of $F_k$.  There is a natural topology on $\mathcal{G}(k)$ relevant to Chabauty's work \cite{chabauty} and effectively used by Grigorchuk in \cite{grigorchuk}.  Following \cite{champetier} (see also \cite{stalder}) it can be described as follows: 
For $G=F_K/N$ and $G'=F_k/N'$ in $\mathcal{G}(k)$, where $N$ and $N'$ are normal subgroups of $F_k$, we set
$$
d(G, G')=\exp (-\max \{r \in \N \mid N\cap B(e,r) = N'\cap B(e,r)\}),
$$
where $B(e,r)$ denotes the ball centered at the identity element $e\in F_k$ with radius $r$ measured by the word metric on $F_k$ with respect to $S$.  Equipped with $d$ above, $\mathcal{G}(k)$ becomes a compact and separable ultrametric space; $(G_m=F_k/N_m) \to (G=F_k/N)$ means that, for any $w \in N$ (resp.~$w \not\in N$), there exists $m_0 \in \N$ such that $w \in N_m$ (resp.~$w \not\in N_m$) for $m\geq m_0$.  

\begin{lem}
\label{lem:limit_in_marked_gr}
 Let $\{G_m\}_{m \in \N}\subset \mathcal{G}(k)$, $\{Y_m\}_{m\in \N}$ and $\{\rho_m\colon G_m \rightarrow \isom{Y_m}\}_{m \in \N}$ be sequences of marked groups, metric spaces and homomorphisms respectively.  
 Suppose that $G_m\to G$ in $\mathcal{G}(k)$ and that there exist a positive real number $\delta$ and a point $p_m \in Y_m$ such that $\max_{s\in S}d_{Y_m}(p_m, \rho_m(s)p_m)$ is uniformly bounded above by $\delta$. 
 Then there is a well-defined homomorphism $\rho\colon G\rightarrow \isom{Y}$ with $\delta(\rho)\leq \delta$, where $Y=(Y,p)=\omega\text{-}\lim (Y_m,p_m)$. 
\end{lem}

\begin{proof}
 Set $G_m=F_k/N_m$ and $G=F_k/N$, where $N_m$ and $N$ are normal subgroups of $F_k$. 
 Let $\tilde \rho_m \colon F_k \rightarrow \isom{Y_m}$ be the lift of $\rho_m$ to $F_k$. 
 Since $\max_{s\in S}d(p_m,\rho_m(s)p_m)$ is uniformly bounded by $\delta$, we obtain a well-defined homomorphism $\tilde \rho\colon F_k \rightarrow \isom{Y}$ with $\delta(\tilde \rho)\leq \delta$.

 Take any $w \in N$.  Then there exists $m_0 \in \N$ such that, for any $m\geq m_0$, $w \in N_m$, which means that $\tilde \rho_m(w)$ is the identity element in $\isom{Y_m}$.  
 Thus $\tilde \rho(w)$ must be the identity element in $\isom{Y}$.  Since this is true for any $w \in \N$, we see that $\tilde \rho$ descends to $G=F_k/N$, and gives rise to a homomorphism $\rho\colon G \rightarrow \isom{Y}$ with $\delta(\rho)\leq \delta$. 
\end{proof}

It is known that marked groups with Kazhdan's property (T) form an open subset of $\mathcal{G}(k)$ (\cite[Theorem 6.7]{shalom}, see also \cite[Theorem 4.3]{stalder}).  It is easy to see that the same is true for groups with the fixed-point property for $\R^n$. 

\begin{prop}
\label{prop:fp(R^n)_open}
 Fix a positive integer $n$.  Then the set of groups with the fixed-point property for $\R^n$ is open in $\mathcal{G}(k)$.  
\end{prop}

\begin{proof}
 Suppose not.  Then there is a sequence $\{G_m\}$ of marked groups such that each $G_m$ does not have the fixed-point property for $\R^n$ and that $G_m \to G$ in $\mathcal{G}(k)$ for $G$ with the fixed-point property for $\R^n$. 
 Then, for each $m\in \N$, we can take a homomorphism $\rho_m \colon G_m \rightarrow \isom{\R^n}$ providing a fixed-point-free action on $\R^n$.  Since $\R^n$ admits scaling, we can scale $\R^n$ so that $\delta(\rho_m)=1$ and find a point $p_m \in \R^n$ so that $\max_{s\in S} d_{\R^n}(p_m,\rho_m(s)p_m)\leq (m+1)/m$.  Note that, since $\R^n$ is proper and $\{(\R^n,p_m)\}$ clearly converges to $(\R^n,p)$ in the Gromov-Hausdorff topology, $\omega\text{-}\lim (\R^n,p_m)$ is isometric to $\R^n$ itself (see \cite[5.51 Remark]{bridson-haefliger}). Since we are assuming $G_m \to G$, we have a homomorphism $\rho \colon G \rightarrow \isom{\R^n}$ by Lemma~\ref{lem:limit_in_marked_gr}.  
 By the definition of the ultralimit, it is clear that $\delta(\rho)=1$. 
 This contradicts our assumption that $G$ has the fixed-point property for $\R^n$. 
\end{proof}

\begin{rem}
  Although marked groups with Property (T) and marked groups with the fixed-point property for $\R^n$ form open subsets in $\mathcal{G}(k)$ respectively, it is not clear whether the same is true for marked groups with the fixed-point property for all finite-dimensional Euclidean spaces or not.   
\end{rem}

The following proposition asserts the lower semicontinuity of $\kappa_S(G,\mathcal{Y})$ on the open subset of $\mathcal{G}(k)$ consisting of groups with fixed-point property for $\R^n$. 

\begin{prop}
\label{prop:lower_semiconti_dim}
 Take a set of real numbers 
 $\mathcal{N}=\{N_{r,\varepsilon} \in \R \mid r>\varepsilon >0\}$. 
 Let $\mathcal{Y}=\mathcal{Y}(\mathcal{N},n)$ be a family of
 geodesically complete $\cat{0}$ spaces whose $(r,\varepsilon)$-capacity
 does not exceed
 $N_{r,\varepsilon} \in \mathcal{N}$ for all $r>\varepsilon> 0$ and
 whose dimension does not exceed $n$. 
 Let $G \in \mathcal{G}(k)$ be a marked group with the fixed-point
 property for $\R^n$.  
 Suppose that $\{G_m\}\subset \mathcal{G}(k)$ is a sequence converging to $G$ in $\mathcal{G}(k)$.
 Then we have 
 $\liminf_{m\to \infty} \kappa_S (G_m,\mathcal{Y})\geq \kappa_S (G,\mathcal{Y})$.    
\end{prop}

\begin{proof}
 First note that we may assume each $G_m$ has the fixed-point
 property for $\R^n$ by Proposition~\ref{prop:fp(R^n)_open}. 
 In particular, there is a constant $\kappa(G_m,\mathcal{Y})>0$ for
 each $G_m$ by Theorem~\ref{thm:almost_fix}.
 We first show that $\inf_{m \in \N}\kappa(G_m,\mathcal{Y})>0$. 
 
 Suppose that $\inf_{m \in \N} \kappa(G_m,\mathcal{Y})=0$.  
 Then, by passing to a subsequence if necessary, we can find a sequence
 of homomorphisms $\rho_m\colon G_m \rightarrow \isom{Y_m}$ with 
 $Y_m \in \mathcal{Y}$ and $\delta(\rho_m)\searrow 0$. 
 Let $p_m \in Y_m$ be a point satisfying
 $\max_{s\in S}d_{Y_m}(\rho_m(s)p_m,p_m)\leq 2\delta(\rho_m)$, and take
 the ultralimit $Y=\omega\text{-}\lim (\frac{1}{\delta(\rho_m)}Y_m,p_m)$. 
 Since $\rho_m$ induces $\bar \rho_m \colon G_m \rightarrow \isom{\frac{1}{\delta(\rho_m)}Y_m}$
 and each $\bar \rho_m$ satisfies $1 \leq \delta(\bar \rho_m) \leq 2$
 uniformly, we obtain a well-defined homomorphism $\rho\colon G \rightarrow \isom{Y}$ by Lemma~\ref{lem:limit_in_marked_gr}, and $1 \leq \delta(\rho) \leq 2$ holds. 
 Noting that $\delta(\rho_m)\searrow 0$ and that the ultralimit does not
 increase the dimension, we see that Proposition~\ref{prop:tits_boundary} implies that 
 $Y \in \mathcal{Y}$ and the compactness of the Tits boundary of $Y$. 
 However, this contradicts Proposition~\ref{prop:bdd_geom_almost_fix}. 
 Thus we conclude that $\inf_{m\in \N} \kappa_S (G_m, \mathcal{Y})>0$. 

 Let $\kappa = \liminf_{m \to \infty} \kappa_S(G_m,\mathcal{Y})$. 
 Take $Y_m' \in \mathcal{Y}$,  $\rho_m' \colon G_m \rightarrow \isom{Y_m'}$, and 
 $p_m' \in Y_m'$ so that $\kappa(G_m,\mathcal{Y})\leq \delta(\rho_m') \leq
 ((m+1)/m)\kappa_S(G_m,\mathcal{Y})$. 
 Passing to a subsequence if necessary, we may assume $\lim_{m\to \infty}\kappa_S(G_m,\mathcal{Y})=\kappa$. 
 Let $Y'=\omega\text{-}\lim (Y_m',p_m')$. 
 Then, since $\omega\text{-}\lim \delta(\rho_m)=\kappa$, we have a well-defined homomorphism 
 $\rho'\colon G \rightarrow \isom{Y'}$ with $\delta(\rho')=\kappa$. 
 By the definition of $\kappa_S(G,\mathcal{Y})$, we see that $\kappa \geq \kappa_S(G,\mathcal{Y})$. 
 The proof is complete. 
\end{proof}

For marked groups with the fixed-point property for all finite-dimensional Euclidean spaces, the same proof as above (with the first paragraph omitted) yields the following. 

\begin{prop}
\label{prop:lower_semiconti}
 Take a set of real numbers 
 $\mathcal{N}=\{N_{r,\varepsilon} \in \R \mid r>\varepsilon >0\}$. 
 Let $\mathcal{Y}=\mathcal{Y}(\mathcal{N},n)$ be a family of
 geodesically complete $\cat{0}$ spaces whose $(r,\varepsilon)$-capacity
 does not exceed
 $N_{r,\varepsilon} \in \mathcal{N}$ for all $r>\varepsilon> 0$. 
 Let $G \in \mathcal{G}(k)$ be a marked group with the fixed-point
 property for all finite-dimensional Euclidean spaces.  
 Suppose that $\{G_m\}_{m \in \N} \subset \mathcal{G}(k)$ 
 is a sequence of marked groups with the fixed-point property for all
 finite-dimensional Euclidean spaces such that $G_m\to G$ in $\mathcal{G}(k)$.
 Then we have 
 $\liminf_{m\to \infty} \kappa_S (G_m,\mathcal{Y})\geq 
 \kappa_S (G,\mathcal{Y})$.    
\end{prop}

\section{A few consequences of the rigidity principle }\label{sec:rigidityconsequences}

\subsection{An example in linear representation theory}
Consider the case 
$$
Y=\mathrm{GL}_n(\mathbb{C})/\mathrm{U}_n(\mathbb{C})
$$ 
and $(\Gamma, S)$ with property $\mathrm F \mathbb R ^{<\infty}$. Any $n$-dimensional complex linear representation of $\Gamma$ induces an isometric action on $Y$. Fixing a point in $Y$ means exactly that it is a unitarizable representation. Theorem \ref{thmD} asserts in this setting that for any linear representation $\rho$ of $\Gamma$ that is not relatively compact, it holds that $\delta(\rho) \geq \kappa_n >0$. In particular, if a unitary representation $\rho_0$ of $\Gamma$ is deformed among the space of all linear representations to $\rho$, small enough such that $\delta(\rho)<\kappa_n$, then $\rho$ must be unitarizable. This is a weaker version of the local rigidity of unitary representations coming from Weil's cohomological criterion \cite{weil} as in \cite{rapinchuk} for property (T) groups. 

\subsection{Visibility space}

\label{ssec:visibility}
The results in Section 5, in particular  Theorem \ref{thmD}, are particularly useful when applied to group actions that are not necessarily properly discontinuous. In this section, we examine isometric actions of torsion groups on visibility spaces. We will deduce a result for torsion groups that complements \cite[Theorem 1.5]{ji-wu_tits_alt} since we do not assume that the action is properly discontinuous, on the other hand, we insist on geodesic completeness, which is not assumed in \cite{ji-wu_tits_alt}. That paper instead uses the Margulis lemma of Breuillard-Green-Tao \cite[Corollary 11.17]{breuillard-green-tao}.

Recall that a CAT(0) space $Y$ is called a \emph{visibility space} if  for every two distinct points in the geometric boundary $\partial Y$ of $Y$, there is a geodesic line whose endpoints are exactly these two points. Typical examples include CAT($-1$) spaces, and more generally Gromov hyperbolic CAT(0) spaces. One may also see \cite{ji-wu-invent} for visibility manifolds with finite-volume quotients of bounded curvatures that are not Gromov hyperbolic. For further background on this subject, we refer to \cite[Chapter II.9]{bridson-haefliger}, \cite{eberlein-oneil} and \cite{ji-wu-invent}. 

Visibility spaces, which is a weakening of hyperbolicity, have discrete Tits boundary, thus this is the opposite case to the case of compact Tits boundary.

\begin{prop}[{\cite[Theorem 3.5]{ji-wu_tits_alt}}]
\label{prop:visibility_fixed_point}
 Let $X$ be a proper visibility $\cat{0}$ space and $\Gamma$ an 
 infinite torsion group. 
 Suppose that $\Gamma$ acts on $X$ with an unbounded orbit. 
 Then there is a point $\xi\in \partial X$ such that any convergent
 unbounded sequence $g_nx_0$ in $\bar X$ converges to $\xi$. 
\end{prop}

\begin{prop}
\label{prop:visibility_torsion_drift0}
 Let $X$ be a proper visibility $\cat{0}$ space and $\Gamma$ a
 countable torsion group equipped with 
 a probability measure $\mu$ with finite first moment.
 Suppose that $\Gamma$ acts on $X$ via a homomorphism
 $\rho\colon \Gamma \rightarrow \isom{X}$ with an unbounded orbit. 
 Then there is a point $\xi\in \partial X$ fixed by $\rho(\Gamma)$, and 
 the drift $l_{\mu}(\rho)$ with respect to $\mu$ is equal to zero.
\end{prop}

\begin{proof}
 Fix $x_0 \in X$. By Proposition~\ref{prop:visibility_fixed_point}, every unbounded
 convergent sequence $\{g_n\circ x_0\}$ in $\bar X$  with $g_n \in \rho(\Gamma)$ converges to a unique point 
 $\xi \in \partial X$.
Taking any $h \in \rho(\Gamma)$, the sequence $\{(h g_n)\circ x_0\}$ is also an unbounded convergent sequence in 
 $\bar X$, and thus converges to $\xi$.  Since $(h g_n) \circ x_0=h(g_n \circ x_0) \to h (\xi)$, we have $h (\xi)=\xi$; therefore, $\xi$ is fixed by $\rho(\Gamma)$.  

 Since $\Gamma$ does not admit nontrivial homomorphism into $\R$, 
 $\rho(\Gamma)$ preserves every horosphere centered at $\xi$. In particular, $g_n x_0$ approaches $\xi$ along the horosphere centered at
 $\xi$ passing through $x_0$ and the Busemann function $b_{\xi}(g_n x_0, x_0)$ is identically zero. 
 However, as mentioned in the proof of
 Theorem~\ref{thm:almost_fix}
 (see \cite[Proposition 4.2]{karlsson-linear} and  
 \cite[\S 3]{karlsson-ledrappier}), if $l_{\mu}(\rho)>0$, it would follow that for a random walk path $g_n$, $b_{\xi}(g_nx_0,x_0)\to -\infty$ as $n\to \infty$, a contradiction. Therefore, $l_{\mu}(\rho)$ must vanish. 
\end{proof}

\begin{thm}
 \label{thm:visibility_fixed_pt}
 Let $Y$ be a geodesically complete visibility space
 with bounded geometry, and let $\Gamma$ be a finitely generated torsion
 group. Whenever $\Gamma$ acts isometrically on $Y$, there is a global fixed point in $Y$.
 
\end{thm}

\begin{proof}
If the orbit is bounded, Cartan’s fixed-point theorem gives a global fixed point. Assume therefore that the orbit is unbounded.
 Since $\Gamma$ is finitely generated, it admits a symmetric probability
 measure $\mu$ supported on a finite generating set. 
 In particular, $\mu$ has finite first and second moments. 
 By Proposition~\ref{prop:visibility_torsion_drift0}, we see that
 $\rho(\Gamma)$ fixes a point in $\partial Y$ and that the drift
 $l_{\mu}(\rho)$ of $\rho(\Gamma)$ with respect to $\mu$ is 
 zero. By Schur's theorem, $\Gamma$ has the fixed-point property for all finite-dimensional Euclidean spaces. 
This, combined with Proposition~\ref{lem:strongly_fixed}, implies that $\rho(\Gamma)$ has
 almost fixed points. Finally, since $Y$ has bounded
 geometry, we can apply Theorem~\ref{thm:almost_fix}, and conclude that
 $\rho(\Gamma)$ has a fixed point in $Y$. 
\end{proof}

\subsection{Compact Tits boundary}
The following is in a sense the opposite case to Theorem \ref{thm:visibility_fixed_pt}:

\begin{cor} \label{cor:Titscompacttorsion}
    Let $Y$ be a geodesically complete $\cat{0}$ space of bounded geometry, and 
    with compact Tits boundary $\partial_T Y$. Any isometric action by a finitely generated torsion group must have a global fixed point in $Y$.
\end{cor}

\begin{proof}
    Recall that every finitely generated torsion group has the fixed-point property for all finite-dimensional Euclidean spaces. By Proposition \ref{prop:bdd_geom_almost_fix}, any isometric action of such a group on $Y$ has almost fixed points. Theorem \ref{thm:almost_fix} then implies that the action has a global fixed point in $Y$.
\end{proof}

\subsection{Dichotomy for the drift of random walks}

An immediate application of section 5 and Theorem \ref{thmD} is the following:
 
\begin{thm}\emph{(Drift dichotomy)}\label{thm:F}
    Let $\Gamma$ be a finitely generated group with
 fixed-point property for $\mathbb R^n$. Let $\rho:\Gamma \rightarrow \isom{Y}$ be an isometric action on an
$n$-dimensional geodesically complete $\cat{0}$ space $Y$ of bounded geometry. Then exactly one of the following holds:
 \begin{itemize}
  \item[{\rm (1)}]  for every symmetric probability measure with finite second moment whose support generates $\Gamma$, the drift $l_{\mu}(\rho)>0$;
  \item[{\rm (2)}] $\rho(\Gamma)$ has a global fixed point in $Y$.
 \end{itemize}
 
\end{thm}
\begin{proof}
Suppose $(1)$ does not hold. That is, the drift $l_{\mu}(\rho)=0$. Then it follows from Proposition \ref{lem:strongly_fixed_dim} that $\Gamma$ has almost fixed points in $Y$. By Theorem \ref{thm:almost_fix} (see also the beginning of Remark \ref{remark-boundedgeometry}) we know that $\rho(\Gamma)$ fixes a point in $Y$, that is, $(2)$ follows. The proof is complete.
\end{proof}

Note the rigidity implication: if the orbit has sublinear growth, $d(y,\gamma (y))=o(||\gamma||_{S})$, then in fact it is bounded, $d(y,\gamma(y))<C$. When the action comes from a real finite-dimensional linear representation $(\pi,V)$, and $Y$ is the corresponding symmetric space $\mathrm{GL}(V)/\mathrm{O}(V)$, positive drift means positive top Lyapunov exponent (the fixed point property implies that the determinant has absolute value $1$) and fixing a point means that $\pi(\Gamma)$ lies in a compact subgroup.



We now prove 
\begin{cor}\emph{(Corollary \ref{cor:Lyapunov})}
Let $\mu$ be a symmetric, Borel probability measure on $SL_n(\mathbb{R})$ with finite second moment, i.e. $\int (\log (\max\{||g||,||g^{-1}||\}))^2 d\mu(g) < \infty $. Assume that the group generated by the support of $\mu$ is a finitely generated group with fixed point property for $\mathbb{R}^{n(n+1)/2-1}$, $n \geq 2$ and that it is unbounded in $SL_n(\mathbb{R})$. Then the top Lyapunov exponent $\lambda_{1}$ is strictly positive. 
\end{cor}
\begin{proof}
We apply Theorem \ref{thm:F} to the $\cat{0}$ space $SL_n(\mathbb{R})/SO_n(\mathbb{R})$. The group $SL_n(\mathbb{R})$ acts by isometries, and the group is unbounded if and only if there is no global fixed point. Denote by $\Gamma$ the group generated by the support of $\mu$. In view of the theorem, the drift is strictly positive. 
In view of $\det (g)=1$,
$$
\lambda_1+\lambda_2 + ...+ \lambda_n =0.
$$
Positive drift means that the Lyapunov vector is non-zero, and hence $\lambda_1 >0$ if the Lyapunov exponents are ordered from largest to smallest.
\end{proof}


\subsection{Geometric Berger-Wang identity}
We now use concepts and notation taken from Breuillard-Fujiwara  \cite{breuillard-fujiwara}. Let $S$ be a finite set of isometries of a metric space $X$. Let
$$
L(S):=\inf_{x\in X} \max_{s\in S} d(x,sx),
$$
which is a notion that has already played an important role above. Recall
$$
L_{\infty}(S)= \lim_{n\rightarrow \infty}L(S^n)/n
$$
and in particular $\lambda(g):=L_{\infty}(\{ g\})$. Finally, $\lambda(S):=\max_{s\in S}\lambda(s)$ and $\lambda_{\infty}(S)=\limsup\limits_{n\rightarrow \infty}\lambda(S^n)/n$.
Breuillard and Fujiwara say that the geometric Berger-Wang identity is satisfied if 
$$
\lambda_{\infty} (S) = L_{\infty}(S).
$$
They show that symmetric spaces of non-compact type, trees and $\delta$-hyperbolic spaces have the Berger-Wang identity for any finite set $S$.

{\begin{cor}\label{cor:BergerWang}
Let $Y$ be a geodesically complete $\cat{0}$ space with bounded geometry. Then any isometric action by a finitely generated torsion group $\Gamma$ on $Y$, which satisfies the geometric Berger-Wang identity for some symmetric finite generating set of $\Gamma$, must fix a point in $Y$.
\end{cor}
}
\begin{proof}
Since the group consists entirely of finite order elements, $\lambda(g)=0$ for every $g\in \Gamma$. Therefore for any finite set $S$, it holds that $\lambda(S^n)=0$ for every $n$, and hence also $\lambda_{\infty}(S)=0$. By the Berger-Wang assumption for a specific generating set $S$, we thus have that $L_{\infty}(S)=0$. This implies in particular that the drift of any probability measure $\mu$ supported on $S$ must be $0$, since
$$
\ell_{\mu}(\rho) \leq \frac{L(S^n)}{n} \rightarrow 0,
$$
as $n\rightarrow \infty$. The conclusion now follows from Theorem \ref{thm:F}.
\end{proof}

\subsection{Deducing the amenable alternative instead from the rigidity principle}
\label{ssec:an-altern-amen}

In this subsection, 
 we explain how the Kazhdan rigidity-principle Theorem \ref{thmD} implies a bounded geometry version of Theorem \ref{thm:amenable-intro} although we seem to need to assume that the space is geodesically complete.
Since a geodesically complete $\cat{0}$ space with bounded geometry is automatically finite-dimensional by  \cite[Theorem 4.9]{cavallucci-sambusetti}, what we prove here is a slightly weaker version of Theorem \ref{thm:amenable-intro}.
However, we think that this implication is interesting and worthwhile to record here, since the geodesic completeness could perhaps one day be removed with an additional idea.

First, we make the following preparation.
\begin{prop}\label{afp-amg}
Let $\Gamma$ be a finitely generated amenable group with
 fixed-point property for all finite-dimensional Euclidean spaces, and let $Y$ be a complete
 $\cat{0}$ space with bounded geometry. Then any isometric action of $\Gamma$ on $Y$ has almost fixed points.    
\end{prop}
\begin{proof}
Since $\Gamma$ is amenable, according to \cite[Main Theorem]{AB98} of Adams-Ballmann, $\Gamma$ either stabilizes a flat subspace or fixes a boundary point $\xi \in \partial Y$. By our assumption, for the former case, $\rho(\Gamma)$ has a fixed point in $Y$, in particular, has almost fixed points; for the latter case, by Lemma \ref{lem:bounded geometry} there exists a CAT(0) space $Z_0$ of bounded geometry together with an isometric action of $\Gamma$ such that $\Gamma$ does not fix any infinity point in $\partial Z_0$ and
\begin{equation*}
 \inf_{x \in Y} \max_{s \in S} d(x, s(x))\leq  \inf_{x \in Z_0} \max_{s \in S} d(x, s(x)), 
\end{equation*}
where $S$ is some finite generating set of $\Gamma$. Applying again \cite[Main Theorem]{AB98} on $\Gamma$ to $Z_0$, we see that $\Gamma$ fixes a flat subspace in $Z_0$. By our assumption, the group $\Gamma$ must fix a point in $Z_0$.  Thus, we have
\begin{equation*}
 \inf_{x \in Y} \max_{s \in S} d(x, s(x))\leq  \inf_{x \in Z_0} \max_{s \in S} d(x, s(x))=0. 
\end{equation*}
This also implies that $\Gamma$ has almost fixed points in $Y$.
\end{proof}

Now we are ready to prove the \emph{a posteriori} slightly weaker version of Theorem \ref{thm:amenable-intro}.

 \begin{thm} \emph{(Amenable alternative, bounded geometry version)} \label{thm:Bboundedgeom}
  Let $\Gamma$ be a finitely generated amenable group. Then exactly one of the following holds:
 \begin{itemize}
  \item[{\rm (1)}]  every isometric action of $\Gamma$ on a geodesically complete $\cat{0}$ space with bounded geometry has a global fixed point;
  \item[{\rm (2)}] $\Gamma$ acts isometrically on  some $\R^n$ $(n\geq 1)$ without a global fixed point. 
 \end{itemize}
 \end{thm}
 \begin{proof}
Suppose $(2)$ does not hold. That is, the group $\Gamma$ has fixed-point property for all finite-dimensional Euclidean spaces. Then $(1)$ follows from Proposition \ref{afp-amg} and Theorem \ref{thm:almost_fix} (see also the beginning of Remark \ref{remark-boundedgeometry}). On the other hand, a fixed-point-free Euclidean action is itself an action on a geodesically complete CAT(0) space of bounded geometry. The proof is complete.
\end{proof}


\section*{Acknowledgments}
We thank Emmanuel Breuillard for pointing out the algebraic characterization of finitely generated amenable groups acting without a global fixed point on Euclidean spaces.

\section*{Funding}
The second-named author was supported by JSPS KAKENHI Grant Number
JP25K00910. The fourth-named author was supported by the Swiss NSF Grants 200020-200400 and 200021-212864, and by the Swedish Research Council Grant 104651320.  The fifth-named author is partially supported by the National Key R \& D Program of China (2025YFA1017500) and NSFC grants No. 12361141813 and 12425107. 


\bibliography{torsion}{}
\bibliographystyle{plain}
\end{document}